\documentclass[12pt]{amsart}

\usepackage{times}
\usepackage[all]{xy}
\usepackage{amscd}
\usepackage{amsfonts}
\usepackage{amsmath}
\usepackage{amssymb}
\usepackage{amsthm}
\usepackage{bbm}
\usepackage{caption}
\usepackage{cite}
\usepackage{color}
\usepackage{extpfeil}
\usepackage{float}
\usepackage{hyperref}
\usepackage{pdflscape}
\usepackage{mathtools}
\usepackage{multirow}
\usepackage{stackengine}
\usepackage{tikz}
\usepackage{tikz-cd}
\usepackage{verbatim}

\DeclareMathOperator{\Hom}{Hom}

\DeclareMathOperator{\Ext}{Ext}

\DeclareMathOperator{\mmod}{mod-}

\DeclareMathOperator{\grmod}{grmod-}

\DeclareMathOperator{\grproj}{grproj-}

\newcommand{\grstab}[2]{\operatorname{grstab}_{#1}\operatorname{-}#2}
\newcommand{\stab}[2]{\operatorname{stab}_{#1}(#2)}
\newcommand{\rproj}[1]{\operatorname{proj-}\hspace{-1.5pt}#1}
\newcommand{\rinj}[1]{\operatorname{inj-}\hspace{-1.5pt}#1}
\newcommand{\lproj}[1]{#1\hspace{-1pt}\operatorname{-proj}}
\newcommand{\linj}[1]{#1\hspace{-1pt}\operatorname{-inj}}
\newcommand{\Proj}[1]{\operatorname{Proj}(#1)}
\newcommand{\Inj}[1]{\operatorname{Inj}(#1)}
\newcommand{\Gproj}[1]{\operatorname{Gproj}(#1)}
\newcommand{\Ginj}[1]{\operatorname{Ginj}(#1)}
\newcommand{\lmod}[1]{#1\hspace{-1pt}\operatorname{-mod}}
\newcommand{\rmod}[1]{\operatorname{mod-}\hspace{-1.5pt}#1}

\DeclareMathOperator{\Z}{\mathbb{Z}}
\DeclareMathOperator{\E}{\mathcal{E}}
\DeclareMathOperator{\A}{\mathcal{A}}
\DeclareMathOperator{\Mimo}{Mimo}
\DeclareMathOperator{\Cok}{Cok}
\DeclareMathOperator{\Ker}{Ker}
\DeclareMathOperator{\MMor}{MMor}
\DeclareMathOperator{\Mor}{Mor}
\DeclareMathOperator{\EMor}{EMor}
\newcommand\xrightarrowtail[2][]{\ensurestackMath{\mathrel{%
  \stackengine{1pt}{%
    \stackengine{0pt}{\rightarrowtail}{\scriptstyle#2}{O}{c}{F}{F}{S}%
  }{\scriptstyle#1}{U}{c}{F}{F}{S}%
}}}

\let\amsamp=&

\newtheorem{theorem}{Theorem}[section]
\newtheorem{lemma}[theorem]{Lemma}
\newtheorem{proposition}[theorem]{Proposition}
\newtheorem{corollary}[theorem]{Corollary}
\theoremstyle{definition}
\newtheorem{definition}[theorem]{Definition}
\theoremstyle{remark}
\newtheorem*{remark}{Remark}

\begin{document}

\title{The $N$-Stable Category}
\author{Jeremy R. B. Brightbill}
\address{UC Santa Barbara, Department of Mathematics, 552 University Rd, Isla Vista, CA 93117}
\email{jbrightbill@math.ucsb.edu}
\urladdr{https://sites.google.com/site/jeremyrbbrightbill}

\author{Vanessa Miemietz}
\address{
School of Mathematics, University of East Anglia, Norwich NR4 7TJ, UK}
\email{v.miemietz@uea.ac.uk}
\urladdr{https://www.uea.ac.uk/~byr09xgu/}

\date{\today}

\maketitle

\begin{abstract}
A well-known theorem of Buchweitz provides equivalences between three categories: the stable category of Gorenstein projective modules over a Gorenstein algebra, the homotopy category of acyclic complexes of projectives, and the singularity category.  To adapt this result to $N$-complexes, one must find an appropriate candidate for the $N$-analogue of the stable category.  We identify this ``$N$-stable category'' via the monomorphism category and prove Buchweitz's theorem for $N$-complexes over a Frobenius exact abelian category.  We also compute the Serre functor on the $N$-stable category over a self-injective algebra and study the resultant fractional Calabi-Yau properties.
\end{abstract}

%%%%%%%%%%%%%%%%%%%%%%%%%%%%%%%%%%%%%%%%%%%%%%%%%%%%%%%%%%%%%%%%%%%%%%%%%%%%%%%%%%%%%%%%%%%%%%%
\section{Introduction}

The notion of $N$-complexes, which goes back to Mayer \cite{Mayer} and was first studied from a homological point of view by Kapranov \cite{Kap} and Dubois-Violette \cite{dubois1998d}, has received significant interest in recent years. As well as having applications in physics via spin gauge fields (see e.g. \cite{Dubois-violette99generalizedcohomology}), they are homologically interesting in their own right (see e.g. \cite{mirmohades}. In addition, they provide the simplest examples of N-differential graded categories, which, for $N$ a prime number, play an important role in categorification at roots of unity, see e.g. \cite{Kh, KhQ, EQ1,EQ2, EQ3}.

In the classical case of $N=2$, which recovers the usual notion of homological algebra, there are numerous deep and important theorems connecting various categories obtained from complexes.  One such example is a celebrated theorem by Buchweitz \cite[Theorem 4.4.1]{buchweitz1987maximal}, which, adapted to the setting of a Frobenius exact abelian category $\A$, provides equivalences between a) $K^{ac}(\Proj{\A})$, the homotopy category of acyclic complexes of projective objects; b) $D^s(\A)$, the singularity category of $\A$ (i.e., the Verdier quotient of the bounded derived category by the thick subcategory of perfect complexes); and c) $\stab{}{\A}$, the stable category of $\A$.  The equivalence between b) and c) was independently proved by Rickard \cite[Theorem 2.1]{Rickard}.

There are obvious $N$-complex analogues of categories a) and b), and an equivalence $K^{ac}_N(\Proj{\A}) \cong D^s_N(\A)$ generalizing Buchweitz was discovered by Bahiraei, Hafezi, and Nematbakhsh \cite{BHN}.  This raises a question:  is there an ``$N$-stable'' category which fills in the missing link in Buchweitz's theorem?  In this paper, we find that the correct object is the stable category of the monomorphism category, $\MMor_{N-2}(\A)$, whose objects are diagrams of $N-2$ successive monomorphisms in $\A$.  The monomorphism category has been intensively studied, particularly for $N=3$  \cite{RS1, ringel2008auslander}, but also for general $N$ \cite{zhang2011monomorphism}.  Monomorphism categories associated to arbitrary species have also recently been studied by \cite{GKKP}.

If $\E$ is an exact category, then $\MMor_{N-2}(\E)$ can be given the structure of an exact category (Proposition \ref{mmor exact}).  If $\E$ is Frobenius, then $\MMor_{N-2}(\E)$ inherits this property (Theorem \ref{frobenius}); in this case, we define the $N$-stable category, $\stab{N}{\E}$ to be the stable category of $\MMor_{N-2}(\E)$.  For a Frobenius exact abelian category $\A$, we prove that there are equivalences $K^{ac}_N(\A) \xrightarrow{\sim} \stab{N}{\A}$ (Theorem \ref{acyclic-stable equivalence}) and $\stab{N}{\A}\xrightarrow{\sim} D^s_N(\A)$ (Theorem \ref{stable-singularity equivalence}) generalizing Buchweitz, demonstrating that the $N$-stable category merits the name.

Classically, the stable category of a finite-dimensional self-injective algebra $A$ provides a rich source of examples of negative or fractional Calabi-Yau categories, a topic of major interest in homological representation theory with connections to many areas of mathematics, see e.g. \cite{keller2008acyclic, keller2008calabi, coelho2020simple, chan2020periodic}.  One might hope the $N$-stable category enjoys similar properties, and in Corollary \ref{CYdim} we prove that if the Nakayama automorphism of $A$ has finite order, then $\stab{N}{A}$ is fractional Calabi-Yau with the denominator parametrized by $N$.

To prove result, we provide an explicit description of the Serre functor on $\stab{N}{A}$ in Theorem \ref{Serre}.  The effect of the Auslander-Reiten translation (from which the Serre functor can easily be derived) on the objects of the stable monomorphism category has already been computed by Ringel and Schmidmeier \cite{ringel2008auslander} for $N=3$ and Xiong, Zhang, and Zhang \cite{xiong2014auslander} for general $N$.  However, utilizing the connection with $N$-complexes, we are able to provide a simpler version of their construction which is also functorial.

The structure of the paper is as follows:  In Section \ref{Definitions and Notation}, we briefly summarize relevant background material while establishing our terminology and notational conventions.  Section \ref{The N-Stable Category} develops the theory of the monomorphism category, culminating in the definition of the $N$-stable category.  The two relevant equivalences of Buchweitz's theorem are generalized in Sections \ref{Acyclic Projective-Injective N-Complexes} and \ref{The N-Singularity Category}.  In Section \ref{Calabi-Yau Properties}, we describe the Serre functor of the $N$-stable category, discuss its Calabi-Yau properties, and provide a worked example.

{\bf Acknowledgements.} V.M. is partially supported by EPSRC grant EP/S017216/1. Part of this research was carried out during a research visit of J.B. to the University of East Anglia, whose hospitality is gratefully acknowledged.  We also thank Paul Balmer for pointing out how the hypothesis of smallness in Proposition \ref{mmor exact, small} could be eliminated.

\section{Definitions and Notation}
\label{Definitions and Notation}
\subsection{Triangulated Categories}

We shall assume the reader is familiar with the basic theory of triangulated categories.  In lieu of a detailed explanation, we give a quick overview of the relevant topics and terminology; for more details, the reader may consult Neeman \cite{neeman2014triangulated} or Gelfand-Manin \cite{gelfand2013methods}.

Let $\mathcal{T}$ be an additive category, and let $\Sigma: \mathcal{T} \xrightarrow{\sim} \mathcal{T}$ be an additive automorphism of $\mathcal{T}$.  We shall call $\Sigma$ the \textbf{suspension functor} on $\mathcal{T}$.  A \textbf{triangle} in $\mathcal{T}$ is any diagram of the form $X \xrightarrow{f} Y \xrightarrow{g} Z \xrightarrow{h} \Sigma X$.  A \textbf{triangulated category} is the data of $\mathcal{T}$, $\Sigma$, and a collection of triangles (called the \textbf{distinguished triangles}), satisfying certain axioms.

If $(\mathcal{T}_1, \Sigma_1)$ and $(\mathcal{T}_2, \Sigma_2)$ are triangulated categories, a \textbf{triangulated functor} $F: \mathcal{T}_1 \rightarrow \mathcal{T}_2$ is the data of an additive functor $F$ and an isomorphism $\phi: F\Sigma_1 \xrightarrow{\sim} \Sigma_2 F$, such that $F$ (together with $\phi$) maps distinguished triangles in $\mathcal{T}_1$ to distinguished triangles in $\mathcal{T}_2$.

Any morphism $f:X \rightarrow Y$ in a triangulated category $\mathcal{T}$ can be extended to a distinguished triangle $X \xrightarrow{f} Y \xrightarrow{g} Z \xrightarrow{h} \Sigma X$.  We refer to $Z$ as the \textbf{cone} of $f$; it is unique up to (non-canoncial) isomorphism.  Similarly, we refer to $X$ as the \textbf{cocone} of $g$.

A full, replete, additive subcategory $\mathcal{S} \subseteq \mathcal{T}$ is said to be a \textbf{triangulated subcategory} if $\mathcal{S}$ is closed under $\Sigma^{\pm 1}$ and the cone of any morphism in $\mathcal{S}$ lies in $\mathcal{S}$.  A triangulated subcategory $\mathcal{S}$ is said to be \textbf{thick} if it is closed under direct summands.  In this case, we can form a new triangulated category $\mathcal{T}/\mathcal{S}$, called the \textbf{Verdier quotient}, with the same objects and suspension functor as $\mathcal{T}$.  There is a natural triangulated functor $\mathcal{T} \rightarrow \mathcal{T}/\mathcal{S}$ which is the identity on objects and whose kernel is precisely $\mathcal{S}$.  $\mathcal{T}/\mathcal{S}$ can also be viewed as the localization of $\mathcal{T}$ with respect to the multiplicative set of morphisms with cone in $\mathcal{S}$, hence morphisms in $\mathcal{T}/\mathcal{S}$ can be expressed in terms of a calculus of left and right fractions.  A triangle in $\mathcal{T}/\mathcal{S}$ is distinguished if and only if it is isomorphic (in $\mathcal{T}/\mathcal{S}$) to a distinguished triangle in $\mathcal{T}$.

\subsection{Serre Duality and Calabi-Yau Categories}

Let $F$ be a field and let $(\mathcal{T}, \Sigma)$ be an $F$-linear, Hom-finite triangulated category.  A \textbf{Serre functor} on $\mathcal{T}$ is an equivalence of triangulated categories $S: \mathcal{T} \xrightarrow{\sim} \mathcal{T}$ together with isomorphisms $\Hom_\mathcal{T}(X, Y) \cong D\Hom_\mathcal{T}(Y, SX)$ which are natural in $X$ and $Y$.  Here $D := \Hom_F(-,F)$ is the $F$-linear duality.

Let $m, l \in \Z$.  We say that $\mathcal{T}$ is \textbf{(weakly) $(m,l)$-Calabi-Yau} if $\mathcal{T}$ has a Serre functor $S$ and there is an isomorphism of functors $S^l \cong \Sigma^m$.  (Elsewhere in the literature, this is often written using the ``fraction'' $\frac{m}{l}$.)  Note that a triangulated category may be $(m,l)$-Calabi-Yau for many different integer pairs $(m, l)$.  If $l = 1$, then we shall simply say that $\mathcal{T}$ is (weakly) $m$-Calabi-Yau.  There is a stronger notion of the Calabi-Yau property, due to Keller \cite{keller2008calabi}, which requires the isomorphism be compatible with the triangulated structure, but our focus will be on the weaker notion.

\subsection{Exact Categories}
\label{Exact Categories}

We recall some basic definitions and terminology regarding exact categories.  For a more comprehensive overview, we refer to B\"uhler \cite{buhler2010exact}.

Let $\E$ be an additive category.  A \textbf{kernel-cokernel pair} in $\E$ is a diagram $X \xrightarrow{i} Y \xrightarrow{p} Z$ such that $i$ is the kernel of $p$ and $p$ is the cokernel of $i$.  Let $\mathcal{S}$ be a collection of kernel-cokernel pairs which is closed under isomorphisms; its elements will be called the \textbf{admissible short exact sequences}.  The kernels in $\mathcal{S}$ are called \textbf{admissible monomorphisms} and the cokernels are called \textbf{admissible epimorphisms}.  If the class of admissible monomorphisms (resp., admissible epimorphisms) contains all identity morphisms, is closed under composition, and is stable under pushouts (resp., pullbacks), we say that the pair $(\E, \mathcal{S})$ is an \textbf{exact category}.  For a more precise statement of the axioms, see \cite[Definition 2.1]{buhler2010exact}.  Note that $(\E, \mathcal{S})$ is exact if and only if $(\E^{op}, \mathcal{S}^{op})$ is exact.  If $(\E, \mathcal{S})$ and $(\E', \mathcal{S}')$ are exact categories, we say an additive functor $F: \E \rightarrow \E'$ is \textbf{exact} if $F(\mathcal{S}) \subseteq \mathcal{S}'$.

If $\E$ is an exact category, we say that a subcategory $\E'$ of $\E$ is \textbf{closed under extensions} if whenever $X \rightarrowtail Y \twoheadrightarrow Z$ is an admissible short exact sequence in $\E$ with $X, Z \in \E'$, then $Y$ is isomorphic to an object in $\E'$.  If $\E'$ is a full, additive subcategory of $\E$ which is closed under extensions, then $\E'$ inherits the structure of an exact category: a kernel-cokernel pair in $\E'$ is admissible if and only if it is admissible in $\E$.  (See \cite[Lemma 10.20]{buhler2010exact}.)  With this inherited structure, we say $\E'$ is a \textbf{fully exact} subcategory of $\E$.

Any additive category can be given the structure of an exact category by defining the split exact sequences to be admissible.  Any abelian category can be given the structure of an exact category by defining every short exact sequence to be admissible.  A small exact category $\E$ can be embedded as a fully exact subcategory of an abelian category \cite[Theorem A.1]{buhler2010exact}.

An object $P$ in an exact category $\E$ is \textbf{projective} if, for every admissible epimorphism $p: Y \twoheadrightarrow Z$ and every morphism $f: P \rightarrow Z$, there exists a lift $g: P \rightarrow Y$ satisfying $f= pg$.  Injective objects are defined dually.  We let $\Proj{\E}$ (resp., $\Inj{\E}$) denote the full subcategory of $\E$ consisting of the projective (resp., injective) objects.  We say $\E$ has \textbf{enough projectives} if for every object $X \in \E$ there exists an admissible epimorphism $P \twoheadrightarrow X$ with $P$ projective; likewise $\E$ has \textbf{enough injectives} if for every object $X$ there is an admissible monomorphism $X \rightarrowtail I$ with $I$ injective.

We define the \textbf{projectively stable category} of $\E$ to be the category $\underline{\E}$ whose objects are those of $\E$ and whose morphisms are given by $\Hom_{\underline{\E}}(X,Y) := \Hom_{\E}(X,Y)/ \mathcal{P}(X,Y)$, where $\mathcal{P}(X,Y)$ is the additive subgroup of morphisms which factor through a projective object.  Dually, we can quotient out by morphisms factoring through injective objects to form the \textbf{injectively stable category} $\overline{\E}$.  If $\Proj{\E} = \Inj{\E}$ and $\E$ has enough projectives and injectives, we say $\E$ is a \textbf{Frobenius exact category}.  In this case, both stable categories coincide and can be given the structure of a triangulated category, which we shall denote by ($\stab{}{\E}, \Omega^{-1})$.    The suspension functor $\Omega^{-1}$ is defined by choosing for each object $X$ an admissible monomorphism $X \rightarrowtail I_X$ into an injective object; $\Omega^{-1}X$ is then defined to be the cokernel of this map.  An admissible short exact sequence $\begin{tikzcd}[column sep=small] X \arrow[r, tail, "f"] & Y \arrow[r, two heads, "g"] & Z\end{tikzcd}$ in $\E$ induces a natural map $h: Z \rightarrow \Omega^{-1}X$ in $\stab{}{\E}$, which gives rise to a triangle $\begin{tikzcd}[column sep=small] X \arrow[r, "f"] & Y \arrow[r, "g"] & Z \arrow[r, "h"] & \Omega^{-1}X\end{tikzcd}$.  The distinguished triangles in $\stab{}{E}$ are those isomorphic to triangles arising in this way.

\subsection{$N$-Complexes}

For a comprehensive introduction to $N$-complexes, we refer the reader to the work of Iyama, Kato, and Miyachi \cite{iyama2017derived}.
Let $\mathcal{A}$ be an additive category, and let $N \ge 2$ be an integer.

An \textbf{$N$-complex} over $\A$ is a sequence of objects of $X^n \in \A$, together with a sequence of morphisms (called \textbf{differentials}) $d_X^n: X^n \rightarrow X^{n+1}$ such that the composition of any $N$ successive differentials is zero.  A morphism $f^\bullet: X^\bullet \rightarrow Y^\bullet$ of $N$-complexes is a sequence of morphisms $f^n: X^n \rightarrow Y^n$ which commute with the differentials.  We denote the category of $N$-complexes over $\A$ by $C_N(\A)$.  As with complexes, we say an $N$-complex $X^\bullet$ is bounded (resp., bounded above, bounded below) if $X^n = 0$ for $|n| \gg 0$ (resp., $n\gg 0$, $n \ll 0$).  We write $C^b_N(\A)$ (resp., $C^-_N(\A)$, $C^+_N(\A)$) for the full subcategory of $C_N(\A)$ consisting of the bounded (resp., bounded above, and bounded below) $N$-complexes.  In the classical case of $N=2$, we shall always omit the subscript.

As an abbreviation, we shall write $d_X^{n,r}$ for the composition $d_X^{n+r-1}\cdots d_X^{n}$ of $r$ successive differentials, beginning with $d_X^n$.  To improve readability, in complex formulae we shall sometimes write $d_X^{\circ, r}$ when the value of $n$ is clear from context.

For $\natural \in \{\text{nothing}, b, +, -\}$, $C^\natural_N(\A)$ carries the structure of a Frobenius exact category, in which the admissible exact sequences are precisely the chainwise split exact sequences of complexes.  For $i \in \mathbb{Z}$, $1 \le k \le N$ and $X \in \A$, let $\mu^i_k(X)$ be the $N$-complex
$$\cdots \rightarrow 0 \rightarrow X \xrightarrow{id_X} \cdots \xrightarrow{id_X} X \rightarrow 0 \rightarrow \cdots$$
with $k$ terms equal to $X$, in positions $i-k+1$ through $i$.  For any $i \in \Z$ and any $X \in \A$, $\mu^i_N(X)$ is projective-injective in $C^\natural_N(\A)$, and every projective-injective object is a direct sum of complexes of this form.  \cite[Theorem 2.1]{iyama2017derived}  The stable category of $C^\natural_N(\A)$ is denoted $K^\natural_N(\A)$ and is called the \textbf{homotopy category} of $N$-complexes over $\A$.

A morphism $f: X^\bullet \rightarrow Y^\bullet$ in $C^\natural_N(\A)$ is \textbf{null-homotopic} if there exists a sequence of morphisms $h^i: X^i \rightarrow Y^{i-N+1}$ satisfying
$$f^i = \sum_{j=1}^{N} d_Y^{i+j-N, N-j} \circ h^{i+j-1} \circ d_X^{i, j-1}$$
The null-homotopic morphisms are precisely those which factor through a projective-injective object \cite[Theorem 2.3]{iyama2017derived}, hence two morphisms of complexes are equal in $K^\natural_N(\A)$ if and only if their difference is null-homotopic.

The suspension functor for the triangulated structure on $K^\natural_N(\A)$ will be denoted by $\Sigma$.  While $\Sigma$ is induced by the Frobenius structure on $C^\natural_N(\A)$, there is a useful explicit description.  Given any $N$-complex $X^\bullet$, for each $n \in \Z$, there are natural morphisms $X^\bullet \rightarrow \mu^n_N(X^n)$ and $\mu^{n+N-1}_N(X^n) \rightarrow X^\bullet$.  By taking direct sums of these morphisms, we obtain chainwise split exact sequences
$$
\begin{tikzcd}
0 \arrow[r] & X^\bullet \arrow[r, tail] & \bigoplus_{n\in \Z} \mu^{n}(X^n) \arrow[r, two heads] & \Sigma X^\bullet \arrow[r] & 0\\
0 \arrow[r] & \Sigma^{-1}X^\bullet \arrow[r, tail] & \bigoplus_{n\in \Z} \mu^{n+N-1}_N(X^n) \arrow[r, two heads] & X^\bullet \arrow[r] & 0
\end{tikzcd}
$$
whose middle terms are projective-injective.  These sequences are functorial in $X^\bullet$ and define $\Sigma$ and $\Sigma^{-1}$ on $C^\natural_N(\A)$.  (Despite the notation, these functors only become mutually inverse on $K^\natural_N(\A)$.)

Let $[n]: C^\natural_N(\A) \rightarrow C^\natural_N(\A)$ denote the standard shift of complexes, with $(X[n])^i = X^{n+i}$.  For $N > 2$, $\Sigma$ does not agree with $[1]$; however, we have the relation $\Sigma^2 \cong [N]$ in $K^\natural_N(\A)$ \cite[Theorem 2.4]{iyama2017derived}.

\subsection{Derived Category of $N$-Complexes}

In this section, let $\A$ be an abelian (not merely additive) category.  Let $N \ge 2$ be an integer.

Let $n \in \Z$, $1 \le r < N$, and $X^\bullet \in C_N(\mathcal{A})$.  Define the \textbf{$r$-th cycle} (resp., \textbf{boundary}, \textbf{homology}) \textbf{group at $n$} to be
\begin{eqnarray*}
Z^n_r(X^\bullet) &:=& ker(d^{n, r}_X)\\
B^n_r(X^\bullet) &:=& im(d^{n-N+r, N-r}_X)\\
H^n_r(X^\bullet) &:=& Z^n_r(X^\bullet)/B^n_r(X^\bullet)
\end{eqnarray*}
It is clear that $B^n_r(X^\bullet)$ is a subobject of $Z^n_r(X^\bullet)$.  Note that our notation convention for $B^n_r(X^\bullet)$ differs from that of \cite{iyama2017derived}.

For $\natural \in \{\text{nothing}, b, +, -\}$, $C^\natural_N(\A)$ is an abelian category, with all limits and colimits computed component-wise.  Given any short exact sequence $\begin{tikzcd} X^\bullet \arrow[r, hook, "f^\bullet"] & Y^\bullet \arrow[r, two heads, "g^\bullet"] & Z^\bullet \end{tikzcd}$ of $N$-complexes, there are long exact sequences in homology
$$\cdots \rightarrow H^n_r(X^\bullet) \xrightarrow{f_*} H^n_r(Y^\bullet) \xrightarrow{g_*} H^n_r(Z^\bullet) \xrightarrow{\delta} H^{n+r}_{N-r}(X^\bullet) \rightarrow \cdots$$
for all $1 \le r < N$.  \cite[Section 3]{dubois1998d}

We say that $X^\bullet \in C_N(\A)$ is \textbf{acyclic} if $H^n_r(X^\bullet) = 0$ for all $n \in \mathbb{Z}$ and $1 \le r < N$.  For $\natural \in \{\text{nothing}, b, +, -\}$, we let $C^{\natural, ac}_N(\A) \subseteq C^{\natural}_N(\A)$ and $K^{\natural, ac}_N(\A) \subseteq K^{\natural}_N(\A)$ denote the full subcategories of acyclic $N$-complexes.  $K^{\natural, ac}_N(\mathcal{A})$ is a thick subcategory of $K^{\natural}_N(\mathcal{A})$ \cite[Proposition 3.2]{iyama2017derived}.  We define the \textbf{derived category of $N$-complexes} to be the Verdier quotient $D^\natural_N(\mathcal{A}):= K^\natural_N(\mathcal{A})/K^{\natural, ac}_N(\mathcal{A})$.  As with ordinary complexes, a short exact sequence in $C_N(\A)$ induces a triangle in $D_N(\A)$ \cite[Proposition 3.7]{iyama2017derived}.

A morphism $s^\bullet$ in $K^\natural_N(\A)$ is a \textbf{quasi-isomorphism} if its cone is acyclic.  This occurs if and only if $H^n_r(s^\bullet)$ is an isomorphism for every $n \in \Z$ and all $1 \le r < N$.

Given an $N$-complex $X^\bullet$ and $n \in N$, define the \textbf{homological truncation of $X^\bullet$ at $n$} to be the complex $\sigma_{\le n}X^\bullet$ given by
$$\sigma_{\le n} X^i = \begin{cases} 0 & i > n \\ Z^i_{n+1-i}(X^\bullet) & n-N+2 \le i \le n \\ X^i & i < n-N+2 \end{cases}$$
with the differential induced by $d_X^\bullet$.  Clearly $H^i_r(\sigma_{\le n}X^\bullet) = 0$ for all $i > n$.  There is a natural inclusion of complexes $\sigma_{\le n}X^\bullet \hookrightarrow X^\bullet$ which induces an isomorphism $H^i_r(\sigma_{\le n}X^\bullet) \cong H^i_r(X^\bullet)$ for all $r$ and all $i \le n$ \cite[Lemma 3.9]{iyama2017derived}.  We define $\sigma_{> n}X^\bullet$ to be the cokernel of this morphism.   

We also define the \textbf{sharp truncation of $X^\bullet$ at $n$} to be the complex $\tau_{\le n}X^\bullet$ which is zero in degrees greater than $n$ and agrees with $X^\bullet$ in degrees less than or equal to $n$.  We define $\tau_{\ge n}X^\bullet$ analogously.

We say $X^\bullet \in D^b_N(\A)$ is \textbf{perfect} if it is isomorphic to a bounded complex of projective objects; let $D^{perf}_N(\A)$ denote the full subcategory of such objects.  In other words, $D^{perf}_N(\A)$ is the essential image of $K^b_N(\Proj{\A})$ in $D^b_N(\A)$.  It is easily verified that $D^{perf}_N(\A)$ is a thick subcategory of $D^b_N(\A)$; we define the \textbf{$N$-singularity category} to be the Verdier quotient $D^s_N(\A) := D^b_N(\A)/D^{perf}_N(\A)$.

\subsection{Gorenstein Algebras}
\label{Gorenstein Algebras}

For a self-contained treatment of the theory of Gorenstein algebras, we refer to the upcoming book by Krause \cite[Chapter 6]{krause2021homological}.  Let $A$ be a finite-dimensional associative algebra over a field $F$.  We shall assume that $A$ is a \textbf{Gorenstein} algebra; that is, $A$ has finite injective dimension as both a left and right $A$-module.  In this case, both the left and right injective dimension of $A$ coincide \cite[Lemma 6.2.1]{krause2021homological}.  If this number is zero, i.e. $A$ is injective as a right and left $A$-module, then we say that $A$ is \textbf{self-injective}; in this case the projective and injective $A$-modules coincide.  

We shall write $\rmod{A}$ and $\lmod{A}$ for the category of finitely-generated right and left $A$-modules, respectively; when we speak of an ``$A$-module'', we shall always mean an object of $\rmod{A}$ unless otherwise specified.  We shall identify $\lmod{A}$ with $\rmod{(A^{op})}$ when convenient.  Given $X \in \rmod{A}$ and $a \in A$, define $r_a: X \rightarrow X$ to be the $F$-linear map given by right multiplication by $a$; for $X \in \lmod{A}$, we similarly define $l_a: X \rightarrow X$ to be left multiplication by $a$.  If $\phi: A \xrightarrow{\sim} A$ is an $F$-algebra automorphism and $X \in \rmod{A}$, define $X_\phi \in \rmod{A}$ by $x\cdot a:= x\phi(a)$, where the right-hand multiplication is done in $X$.

We shall abbreviate $\Proj{\rmod{A}}$ by $\rproj{A}$, and $\Inj{\rmod{A}}$ by $\rinj{A}$; for left modules we use the abbreviations $\lproj{A}$ and $\linj{A}$.  We say that $X\in \rmod{A}$ is \textbf{Gorenstein projective} (resp., \textbf{Gorenstein injective}) if $\Ext^i_A(X, A) = 0$ (resp., $\Ext^i_A(DA, X) = 0$) for all $i >0$, where $D = \Hom_F(-, F)$ is the $F$-linear duality.  We denote the full subcategory of all Gorenstein projective (resp., Gorenstein injective) modules by $\Gproj{A}$ (resp., $\Ginj{A}$).

$\Gproj{A}$ forms a fully exact subcategory of the abelian category $\rmod{A}$.  In fact, $\Gproj{A}$ is a Frobenius category whose projective-injective objects are precisely $\rproj{A}$ \cite[Theorem 6.2.5]{krause2021homological}.  $D$ restricts to an equivalence $\Gproj{A}^{op} \xrightarrow{\sim} \Ginj{A^{op}}$, hence $\Ginj{A}$ is also Frobenius exact and its projective-injective objects are precisely $\rinj{A}$.  When $A$ is self-injective, note that $\Gproj{A} = \rmod{A} = \Ginj{A}$.

The \textbf{Nakayama functor} $\nu_A: \rmod{A} \rightarrow \rmod{A}$ is the composition $\nu_A := D\Hom_A(-, A) \cong X\otimes_A DA$.  The functor $\Hom_A(-,A)$ restricts to an exact duality $\Gproj{A} \xrightarrow{\sim} \Gproj{A^{op}}$ \cite[Lemma 6.2.2]{krause2021homological}, hence $\nu_A$ defines an exact equivalence $\Gproj{A} \xrightarrow{\sim} \Ginj{A}$ which descends to a triangulated equivalence of the respective stable categories.

If $A$ is self-injective, then $\nu_A$ is an exact autoequivalence of both $\rmod{A}$ and $\lmod{A}$ and preserves projective-injectives; in this case, $\nu_A$ lifts to $D^b_N(A)$ and descends to $D^s_N(A)$.  There is an $F$-algebra automorphism $\phi_{A}$, called the \textbf{Nakayama automorphism}, such that $\nu_A(X) = X_{\phi_A}$.  The Nakayama automorphism is unique up to a choice of inner automorphism.

%%%%%%%%%%%%%%%%%%%%%%%%%%%%%%%%%%%%%%%%%%%%%%%%%%%%%%%%%%%%%%%%%%%%%%%%%%%%%%%%%%%%%%%%%%%%%%%%%%%%%%%%%%%%%%%%%%%%%%%%%%%%%%%%%%%%%%%%%%%%%%%%%%%%%%%%%%%

\section{The N-Stable Category}
\label{The N-Stable Category}

\subsection{The Monomorphism Category}

Throughout this section, let $(\E, \mathcal{S})$ be an exact category.

For any integer $k\ge 1$, let $[[k]]$ denote the category corresponding to the poset $\{1< \cdots < k\}$.  For any $k \ge 0$, let $\Mor_{k}(\E)$ denote the category $\E^{[[k+1]]}$ of functors from $[[k+1]]$ to $\E$.  Namely, the objects of $\Mor_k(\E)$ are diagrams $(X_\bullet, f_\bullet) = X_1 \xrightarrow{f_1} \cdots \xrightarrow{f_k} X_{k+1}$ of $k$ composable morphisms in $\E$.  $\Mor_k(\E)$ carries a natural structure of an exact category, in which the class of admissible exact sequences is $\mathcal{S}^{[[k+1]]}$.  That is, $X_\bullet \rightarrowtail Y_\bullet \twoheadrightarrow Z_\bullet$ is admissible if and only if $X_i \rightarrowtail Y_i \twoheadrightarrow Z_i$ is admissible in $\E$ for each $1 \le i \le k+1$.  (See B\"uhler, \cite[Example 13.11]{buhler2010exact}.)  As in all diagram categories, small limits and colimits in $\Mor_k(\E)$ are computed component-wise and exist if and only if the component-wise limits and colimits exist (see, for instance, \cite[Proposition 2.15.1]{borceux1994handbook}).  Note that $\Mor_0(\E)$ recovers $\E$ as an exact category.

Mimicking our notation for $N$-complexes, given $(X_\bullet, f_\bullet) \in \Mor_k(\E)$ we will write $f_i^j := f_{i+j-1} \cdots f_i$ for the composition of $j$ successive maps in $f_\bullet$, beginning with $f_i$.  We shall let $f_i^0$ denote the identity map on $X_i$.

\begin{definition}
Let $(\E, \mathcal{S})$ be an exact category.  Let $k \ge 0$.  Let the \textbf{monomorphism subcategory} $\MMor_{k}(\E)$ be the full subcategory of $\Mor_k(\E)$ consisting of objects of the form $$X_1 \xrightarrowtail{\iota_1} X_2 \xrightarrowtail{\iota_2} \cdots \xrightarrowtail{\iota_k} X_{k+1}$$ where each $\iota_j$ is an admissible monomorphism in $\E$.

An \textbf{admissible short exact sequence} in $\MMor_k(\E)$ is any short exact sequence $X_\bullet \rightarrowtail Y_\bullet \twoheadrightarrow Z_\bullet$ which is admissible in $\Mor_k(\E)$.  Write $\MMor_k(\mathcal{S})$ for the class of admissible short exact sequences in $\MMor_k(\E)$.
\end{definition}

\begin{remark}
We could also define the \textbf{epimorphism subcategory} $\EMor_k(\E)$ to be the analogous subcategory of $\Mor_{k}(\E)$ in which every morphism appearing in the diagram is an admissible epimorphism in $\E$.  By again declaring all component-wise admissible exact sequences to be admissible, we obtain a candidate structure of exact category on $\EMor_k(\E)$.  There is a natural equivalence of categories between $\EMor_k(\E)$ and $\MMor_k(\E^{op})$ which preserves their candidate exact structures.  Thus dual versions of all results in this section apply to $\EMor_k(\E)$; the reader can easily formulate the precise statements.
\end{remark}

Our goal is to show that the above definitions give $\MMor_k(\E)$ the structure of an exact category.  The result is straightforward in the case of abelian categories.

%%%%%%%%%%%%%%%%%%%%%%%%%%%
\begin{proposition}
\label{mmor exact, abelian case}
Let $\mathcal{A}$ be an abelian category.  Then $\MMor_k(\A)$ is closed under extensions in $\Mor_k(\A)$.  In particular, $\MMor_k(\mathcal{A})$ is a fully exact subcategory of $\Mor_k(\A)$.
\end{proposition}

\begin{proof}
Suppose we have a short exact sequence $X_\bullet \hookrightarrow Y_\bullet \twoheadrightarrow Z_\bullet$, where $(X_\bullet, \alpha_\bullet), (Z_\bullet, \beta_\bullet) \in \MMor_k(\A)$ and $(Y_\bullet, \beta_\bullet) \in \Mor_k(\A)$.  By the Snake Lemma, for each $1 \le i \le k$ we have a short exact sequence
$$\begin{tikzcd} 0 \arrow[r] & ker(\alpha_i) \arrow[r] & ker(\beta_i) \arrow[r] & ker(\gamma_i)\end{tikzcd}$$
Since $ker(\alpha_i) = ker(\gamma_i) = 0$, it follows that $ker(\beta_i) = 0$ and $\beta_i$ is a monomorphism for all $i$.  Thus $(Y_\bullet, \beta_\bullet) \in \MMor_k(\A)$, and so $\MMor_k(\A)$ is closed under extensions.

It is clear $\MMor_k(\A)$ is a full additive subcategory of $\Mor_k(\A)$, and that the candidate exact structure on $\MMor_k(\A)$ agrees with that inherited from $\Mor_k(\A)$.  Thus $\MMor_k(\A)$ is a fully exact subcategory of $\Mor_k(\A)$.
\end{proof}
%%%%%%%%%%%%%%%%%%%%%%%%%%%

%%%%%%%%%%%%%%%%%%%%%%%%%%%
\begin{proposition}
\label{mmor exact, small}
Let $\E$ be a small exact category.  Then $\MMor_k(\E)$ is exact.
\end{proposition}

\begin{proof}
Since $\E$ is small, by \cite[Theorem A.1]{buhler2010exact}, there exists an abelian category $\A$ and a fully faithful exact functor $\iota: \E \rightarrow \A$ such that $\iota$ reflects exactness and $\E$ is closed under extensions in $\A$.  It is clear that $\iota$ induces an additive functor $\iota_*: \Mor_k(\E) \rightarrow \Mor_k(\A)$, which remains fully faithful and sends objects of $\MMor_k(\E)$ to $\MMor_k(\A)$.  Thus we may view $\MMor_k(\E)$ as a full, additive subcategory of $\MMor_k(\A)$; accordingly, we will suppress mention of the functor $\iota$ in our notation going forward.

We claim that $\MMor_k(\E)$ is closed under extensions in $\MMor_k(\A)$, hence is a fully exact subcategory.  Let $\begin{tikzcd} X_\bullet \arrow[r, tail, "f_\bullet"] & Y_\bullet \arrow[r, two heads, "g_\bullet"] & Z_\bullet\end{tikzcd}$ be a short exact sequence in $\MMor_k(\A)$, with $(X_\bullet, \alpha_\bullet), (Z_\bullet, \gamma_\bullet) \in \MMor_k(\E)$.  We must show that $(Y_\bullet, \beta_\bullet) \in \MMor_k(\E)$.

For each $i$, we have a short exact sequence $\begin{tikzcd} X_i \arrow[r, hook, "f_i"] & Y_i \arrow[r, two heads, "g_i"] & Z_i\end{tikzcd}$ in $\A$.  Thus $Y_i\in \E$, since $\E$ is closed under extensions.  Since the inclusion functor $\iota: \E \rightarrow \A$ reflects exactness, the above short exact sequence is admissible in $\E$.

It remains to show that the monomorphisms $\beta_i$ are admissible in $\E$.  Consider the diagram
\begin{eqnarray*}
\begin{tikzcd}
X_i \arrow[d, tail, "f_i"] \arrow[r, tail, "\alpha_i"] & X_{i+1} \arrow[d, tail, "f_{i+1}"] \arrow[r, two heads] & coker(\alpha_i) \arrow[d, dashed, hook, "\phi"]\\
Y_i \arrow[d, two heads, "g_i"] \arrow[r, hook, "\beta_i"] & Y_{i+1} \arrow[d, two heads, "g_{i+1}"] \arrow[r, two heads]& coker(\beta_i) \arrow[d, dashed, two heads, "\psi"]\\
Z_i \arrow[r, tail, "\gamma_i"]& Z_{i+1} \arrow[r, two heads]& coker(\gamma_i)
\end{tikzcd}
\end{eqnarray*}
The first two columns are admissible and exact in $\E$ by the above remarks; we construct the third column by applying the Snake Lemma and deduce that it is a short exact sequence in $\A$.  The monomorphisms $\alpha_i$ and $\gamma_i$ are admissible in $\E$, hence $coker(\alpha_i), coker(\gamma_i) \in \E$.  Since $\E$ is closed under extensions and $\iota$ reflects exactness, $coker(\beta_i) \in \E$ and the third column is an admissible short exact sequence in $\E$.  Thus all the objects in the second row lie in $\E$, hence the second row is an admissible short exact sequence in $\E$.  In particular, $\beta_i$ is an admissible monomorphism in $\E$.  Thus $(Y_\bullet, \beta_\bullet) \in \MMor_k(\E)$.

It remains to show that the structure of exact category which $\MMor_k(\E)$ inherits from $\MMor_k(\A)$ agrees with the original exact structure, i.e. that which it inherited from $\Mor_k(\E)$.  This follows immediately from the fact that $\iota$ is exact and reflects exactness.
\end{proof}
%%%%%%%%%%%%%%%%%%%%%%%%%%%%

Since verifying the axioms of an exact category only involves working with finitely many objects at a time, the smallness hypothesis in the previous proposition can be removed.

\begin{lemma}
\label{exact categories are locally small}
Let $(\E, \mathcal{S})$ be an exact category, and let $E \subseteq Ob(\E)$ be a set of objects.  Then there exists a small full subcategory $\E'$ of $\E$ containing $E$, such that $(\E', \mathcal{S}')$ is an exact category, where $\mathcal{S}'$ is the set of all kernel-cokernel pairs in $\mathcal{S}$ whose objects lie in $\E'$.
\end{lemma}

\begin{proof}
Given any full subcategory $T$ of $\E$, let $C(T)$ (resp., $K(T)$) be the full subcategory of $\E$ consisting of the objects $coker(f)$ (resp., $ker(f)$), where $f$ ranges over all morphisms in $T$ which are admissible monomorphisms (resp., epimorphisms) in $\E$.  In this definition we make a single choice of $coker(f)$ or $ker(f)$ for each morphism $f$, hence $C(T)$ and $K(T)$ are small if $T$ is.  For each $X \in Ob(T)$, we choose $X$ to be the representative of both $ker(X \rightarrow 0)$ and $coker(0 \rightarrow X)$, so that $T$ is a full subcategory of both $K(T)$ and $C(T)$.  Finally, it is easily checked that if $T$ is an additive subcategory of $\E$, then so are $C(T)$ and $K(T)$.

For any finite sequence $X_1, \cdots, X_n$ of objects in $E$, choose one object of $\E$ isomorphic to $\bigoplus_{i=1}^n X_i$, and let $E_0$ be full subcategory of $\E$ consisting of all chosen objects.  Then $E_0$ is a small additive subcategory of $\E$ which can be chosen to contain $E$.  For each $i > 0$, inductively define $E_i := K(C(E_{i-1}))$, and let $\E' := \bigcup_{i=0}^\infty E_i$.  It is clear that $\E'$ is a small additive subcategory of $\E$ containing $E$.

It remains to show that $(\E', \mathcal{S}')$ is an exact category.  It is immediate that all identity morphisms are admissible epimorphisms and monomorphisms.  If $f$ and $g$ are two composable admissible monomorphisms in $E_i$, then $cok(f \circ g) \in E_{i+1}$ hence $f\circ g$ is an admissible monomorphism in $\E'$; by a dual argument, composition of admissible epimorphisms in $\E'$ also remain admissible.  Similarly, if $f: X \rightarrowtail Y$ and $g: X \rightarrow Z$ are morphisms in $E_i$ with $f$ an admissible monomorphism, then by \cite[Proposition 2.12]{buhler2010exact} the pushout $P$ of $f$ along $g$ in $\E$ fits into admissible exact sequences
\begin{eqnarray*}
\begin{tikzcd}[ampersand replacement=\&]
X \arrow[r, tail, "\begin{bmatrix} f \\ g\end{bmatrix}"] \& Y \oplus Z \arrow[two heads]{r}{\begin{bmatrix} g' & f' \end{bmatrix}}  \& P\\
Z \arrow[r, tail, "f'"] \& P  \arrow[r, two heads] \& coker(f)
\end{tikzcd}
\end{eqnarray*}
The first sequence shows that, up to isomorphism, $P \in E_{i+1}$.  Since $coker(f) \in E_{i+1}$, we have that $f'$ is an admissible monomorphism in $\E'
$.  By a dual argument, pull-backs preserve admissible epimorphisms in $\E'$. 
\end{proof}

\begin{proposition}
\label{mmor exact}
Let $\E$ be an exact category.  Then $\MMor_k(\E)$ is exact.
\end{proposition}

\begin{proof}
We let $\mathcal{S}$ denote the class of admissible exact sequences in $\E$.  If $E \subseteq \E$ is any finite set of objects, let $(\E', \mathcal{S}')$ be the small exact category containing $E$ constructed in Proposition \ref{exact categories are locally small}.  Then the inclusion functor $\E' \hookrightarrow \E$ is exact and induces a fully faithful functor $\MMor_k(\E') \hookrightarrow \MMor_k(\E)$ which maps $\MMor_k(\mathcal{S}')$ into $\MMor_k(\mathcal{S})$.  By Proposition \ref{mmor exact, small}, ($\MMor_k(\E'), \MMor_k(\mathcal{S}'))$ is an exact category.

To verify the exact category axioms, we need work only with finitely many objects of $\E$ at a time, hence exactness of $\MMor_k(\E)$ can be verified inside $\MMor_k(\E')$.  For instance, to verify that the push-out of the admissible monomorphism $f_\bullet: X_\bullet \rightarrowtail Y_\bullet$ along $g_\bullet: X_\bullet \rightarrow Z_\bullet$ is an admissible monomorphism, let $E= \{X_i, Y_i, Z_i \mid 1 \le i \le k+1\}$.  Then the pushout of $f_\bullet$ along $g_\bullet$ exists and is an admissible monomorphism in $MMor_k(\E')$, hence in $MMor_k(\E)$.  Verification of the other axioms is analogous.
\end{proof}

We close this section by providing a convenient intrinsic description of the admissible monomorphisms and epimorphisms in the monomorphism category of an abelian category.

%%%%%%%%%%%%%%%%%%%%%%%%%%%%
\begin{proposition}
\label{epi mono classification}
Let $\A$ be an abelian category and let $f_\bullet : (X_\bullet, \alpha_\bullet) \rightarrow (Y_\bullet, \beta_\bullet)$ be a morphism in $\MMor_k(\A)$.  $f_\bullet$ is an admissible epimorphism if and only if each $f_i$ is an epimorphism.  $f_\bullet$ is an admissible monomorphism if and only if each $f_i$ is a monomorphism and each sub-diagram
\begin{eqnarray*}
\begin{tikzcd}
X_i \arrow[r, hook, "\alpha_i"] \arrow[d, "f_i", hook] & X_{i+1} \arrow[d, "f_{i+1}", hook]\\
Y_i \arrow[r, hook, "\beta_i"] & Y_{i+1}
\end{tikzcd}
\end{eqnarray*}
forms a pullback square in $\A$.
\end{proposition}

\begin{proof}
If $f_\bullet$ is an admissible epimorphism, it follows immediately that each $f_i$ is epic.  Conversely, if each $f_i$ is an epimorphism, then $f_\bullet$ is an epimorphism in $\Mor_k(\A)$, hence it has a kernel $(K_\bullet, \iota_\bullet)$.  To prove that $f_\bullet$ is an admissible epimorphism, we must show $K_\bullet \in \MMor_k(\A)$.  We have a commutative diagram 
\begin{eqnarray*}
\begin{tikzcd}
K_i \arrow[r, "\iota_i"] \arrow[d, hook] & K_{i+1} \arrow[d, hook]\\
X_i \arrow[r, hook, "\alpha_i"] & X_{i+1}
\end{tikzcd}
\end{eqnarray*}
from which it is clear that $\iota_i$ is a monomorphism.  Thus $K_\bullet \in \MMor_k(\A)$.

If $f_\bullet$ is an admissible monomorphism, then we have a short exact sequence $\begin{tikzcd} X_\bullet \arrow[r, hook, "f_\bullet"] & Y_\bullet \arrow[r, two heads, "g_\bullet"] & Z_\bullet \end{tikzcd}$ with $(Z_\bullet, \gamma_i) \in \MMor_k(\A)$.  It follows immediately that each $f_i$ is a monomorphism.  To show $X_i$ is a pullback, consider the commutative diagram with exact columns
\begin{eqnarray*}
\begin{tikzcd}
T \arrow[drr, bend left, "\psi"] \arrow[ddr, bend right, "\phi"] \arrow[dr, "\eta", dashed]& &\\
& X_i \arrow[r, hook, "\alpha_i"] \arrow[d, hook, "f_i"] & X_{i+1} \arrow[d, hook, "f_{i+1}"]\\
& Y_i \arrow[r, hook, "\beta_i"] \arrow[d, two heads, "g_i"] & Y_{i+1} \arrow[d, two heads, "g_{i+1}"]\\
& Z_i \arrow[r, hook, "\gamma_i"] & Z_{i+1}
\end{tikzcd}
\end{eqnarray*}
where $\psi$ and $\phi$ satisfy $f_{i+1}\psi = \beta_i \phi$.  Postcomposing this equation with $g_{i+1}$, we see that $0 = g_{i+1}f_{i+1}\psi = g_{i+1}\beta_i \phi = \gamma_i g_i \phi$.  Since $\gamma_i$ is a monomorphism, $g_i \phi = 0$.  By exactness of the first column there exists a unique $\eta: T \rightarrow X_i$ such that $\phi = f_i \eta$.  An easy diagram chase yields $f_{i+1}\psi = f_{i+1} \alpha_i \eta$.  Since $f_{i+1}$ is a monomorphism, we have $\psi  = \alpha_i \eta$, hence the top square is a pullback.

Conversely, assume each $f_i$ is a monomorphism and each square in $f_\bullet$ is a pullback.  Let $(Z_\bullet, \gamma_\bullet)$ be the cokernel of $f_\bullet$ in $\Mor_k(\A)$.  We must show that $Z_\bullet \in \MMor_k(\A)$, i.e that each $\gamma_i$ is monic.  We shall construct the following commutative diagram:
\begin{eqnarray*}
\begin{tikzcd}
& X_i \arrow[r, hook, "\alpha_i"] \arrow[d, hook, "f_i", swap] \arrow[dr, phantom, near start, "\lrcorner"] & X_{i+1} \arrow[d, hook, "f_{i+1}"]\\
T' \arrow[d, two heads, "g'", swap] \arrow[r, "\phi'"] \arrow[dr, phantom, near start, "\lrcorner"] \arrow[urr, bend left = 100, "\psi", dashed] \arrow[ur, "\eta", dashed] & Y_i \arrow[r, hook, "\beta_i"] \arrow[d, two heads, "g_i"] & Y_{i+1} \arrow[d, two heads, "g_{i+1}"]\\
T \arrow[r, "\phi"] \arrow[rr, bend right, "0", swap] & Z_i \arrow[r, "\gamma_i"] & Z_{i+1}
\end{tikzcd}
\end{eqnarray*}

We start with the rightmost two squares, which are commutative with exact columns.  To show $\gamma_i$ is a monomorphism, consider $\phi: T \rightarrow Z_i$ such that $\gamma_i \phi = 0$.  Let $T'$ be the pullback of $\phi$ along $g_i$; since $g_i$ is an epimorphism, so is $g'$.  We have that $g_{i+1}\beta_i \phi' = \gamma_i \phi g' = 0$, so by exactness of the right column $\beta_i \phi' = f_{i+1} \psi$ for some $\psi:  T' \rightarrow X_{i+1}$.  Since the top right square is a pullback, we obtain a morphism $\eta: T' \rightarrow X_i$ making the diagram commute.  It follows that $\phi g' = g_i f_i \eta = 0$, hence $\phi = 0$.  Thus $\gamma_i$ is a monomorphism, $Z_\bullet \in \MMor_k(\mathcal{A})$, and $f_\bullet$ is an admissible monomorphism.
\end{proof}

\begin{remark}
Both of the above criteria can fail when $\A$ is not abelian.\\
1)  Let $A$ be the path algebra of the A3 Dynkin quiver $1 \leftarrow 2 \rightarrow 3$, and let $S_i$ be the simple module corresponding to vertex $i$.  Let $\E$ be the full subcategory of $\rmod{A}$ obtained by removing all objects isomorphic to $S_3$.  $\E$ is a full additive subcategory of $\rmod{A}$ which is closed under extensions and is therefore a fully exact subcategory of $\rmod{A}$.

Consider the objects $X_\bullet = S_1 \hookrightarrow \begin{tikzcd}[row sep=-0.6 em, column sep = -0.8 em] & S_2 & \\ S_1 & & S_3 \end{tikzcd}$ and $Y_\bullet = 0 \hookrightarrow S_2$ in $\MMor_1(\E)$.  There is an obvious component-wise epimorphism $f_\bullet: X_\bullet \rightarrow Y_\bullet$ with kernel $K_\bullet = S_1 \hookrightarrow S_1 \oplus S_3$.  Since $S_3$ is not an object of $\E$, the monomorphism defining $K_\bullet$ has no cokernel in $\E$, hence is not admissible.  Thus $K_\bullet \notin \MMor_1(\E)$, and so $f_\bullet$ is not a distinguished epimorphism in this category.

An additive category is \textbf{weakly idempotent complete} if every split monomorphism has a cokernel (or, equivalently, every split epimorphism has a kernel).  Using the dual of \cite[Corollary 7.7]{buhler2010exact}, one can show that if $\E$ is weakly idempotent complete, then the epimorphism criterion in the above proposition holds.
\\
2)  Let $B$ be the path algebra of the D4 Dynkin quiver $\begin{tikzcd}[row sep=small, column sep=0.3 em] 1 \arrow[dr] & 2 \arrow[d] & 3 \arrow[dl] \\ & 4 & \end{tikzcd}$, and let $S_i$ be the simple module corresponding to vertex $i$.  Let $\E$ be the full subcategory of $\rmod{B}$ obtained by removing all objects isomorphic to $S_3$.  As before, $\E$ is a fully exact subcategory of $\rmod{B}$.

Let $X_\bullet = S_4 \hookrightarrow \begin{tikzcd}[row sep=-0.6 em] S_1 \\ S_4 \end{tikzcd}$ and $Y_\bullet = \begin{tikzcd}[row sep=-0.6 em] S_2 \\ S_4 \end{tikzcd} \hookrightarrow \begin{tikzcd}[row sep=-0.6 em, column sep = -0.6 em] S_1 & S_2 & S_3 \\ & S_4 & \end{tikzcd}$ in $\MMor_1(\E)$.  The natural inclusions $f_i: X_i \hookrightarrow Y_i$ induce a monomorphism $f_\bullet : X_\bullet \hookrightarrow Y_\bullet$ in $\MMor_1(\E)$, and it is clear that the commutative square defined by $f_\bullet$ is a pullback.  The cokernel of $f_\bullet$ is $Z_\bullet = S_1 \hookrightarrow S_1 \oplus S_3$.  Once again, $S_3 \notin \E$, hence the monomorphism defining $Z_\bullet$ is not admissible in $\E$ and so $Z_\bullet \notin \MMor_1(\E)$.  Therefore $f_\bullet$ is not an admissible monomorphism in $\MMor_1(\E)$.

If every monomorphism in $\E$ is admissible, then the proof of monomorphism criterion in the above proposition holds with minimal changes.  This is a very strong hypothesis; we do not know if there is a weaker one.
\end{remark}

%%%%%%%%%%%%%%%%%%%%%%%%%%%%%%%%%%%%%%%%%%%%%%%%%%%%%%%%%%%%%%%%%%%%%%%%%%%%

\subsection{Projective and Injective Objects}
\label{Projective and Injective Objects}

We shall classify the projective and injective objects of $\MMor_k(\mathcal{E})$.  It will be convenient to introduce some notation.

\begin{definition}
\label{fundamental objects}
For $X \in \E$ and $1 \le i \le k+1$, let $\chi_i(X)_\bullet \in \Mor_k(\E)$ be given by $0 \rightarrow \cdots \rightarrow 0 \rightarrow X \xrightarrow{id_X} \cdots \xrightarrow{id_X} X$, where the first $i-1$ objects are $0$, and the first $X$ is in position $i$.
\end{definition}

The following lemma, adapted from the proof of \cite[Proposition 2.12]{buhler2010exact}, will be useful.

\begin{lemma}[B\"uhler \cite{buhler2010exact}]
\label{direct sum lemma}
Let $\iota: X \rightarrowtail Y$ be an admissible monomorphism in $\E$, and let $f: X \rightarrow Z$ be any morphism.  Then $\begin{bmatrix} \iota \\ f \end{bmatrix}: X \rightarrowtail Y \oplus Z$ is an admissible monomorphism.  Dually, if $p: Y \twoheadrightarrow W$ is an admissible epimorphism and $g: Z \rightarrow W$ is any morphism, then $\begin{bmatrix} p & g \end{bmatrix}: Y \oplus Z \twoheadrightarrow W$ is an admissible epimorphism.
\end{lemma}

\begin{proof}
We can factor $\begin{bmatrix} \iota \\ f \end{bmatrix}$ as the composition
\begin{eqnarray*}
\begin{tikzcd}[ampersand replacement=\&, column sep=huge]
X \arrow[r, tail, "\begin{bmatrix} id_X \\ 0\end{bmatrix}"] \& X \oplus Z \arrow[r, swap, "\sim"] \arrow{r}{\begin{bmatrix}id_X & 0 \\ f & id_Z \end{bmatrix}} \& X \oplus Z \arrow[tail]{r}{\begin{bmatrix}\iota & 0 \\ 0 & id_Z \end{bmatrix}} \& Y \oplus Z
\end{tikzcd}
\end{eqnarray*}
Split monomorphisms and isomorphisms are admissible monomorphisms, as is the direct sum of two admissible monomorphisms \cite[Proposition 2.9]{buhler2010exact}.  Thus $\begin{bmatrix} \iota \\ f \end{bmatrix}$ is the composition of three admissible monomorphisms.

The proof of the second statement is dual.
\end{proof}
%%%%%%%%%%%%%%%%%%%%%%%%%%%%%%%%%%%%%%%

\begin{proposition}
\label{inj/proj characterization}
Let $\E$ be an exact category.  Then $(I_\bullet, \iota_\bullet) \in \MMor_k(\E)$ is injective (resp., projective) if and only if each $I_i$ is injective (resp., projective) in $\E$ and each $\iota_i$ is split.
\end{proposition}

\begin{proof}
Take $(I_\bullet, \iota_\bullet) \in \MMor_k(\E)$ with each $I_i$ injective and each $\iota_i$ split.  Then we have $I_\bullet \cong \bigoplus_{i=1}^{k+1} \chi_i(I_i')_\bullet$, where $I_1' = I_1$ and $I_i' = coker(\iota_{i-1})$ for $i > 1$.  Thus it suffices to show that $\chi_i(I)_\bullet$ is injective for every injective object $I$ and each $1 \le i \le k+1$.

Fix $I$ and $i$ and suppose $f_\bullet:  \chi_i(I)_\bullet \rightarrowtail (X_\bullet, \alpha_\bullet)$ is an admissible monomorphism; we shall define a retraction $r_\bullet$.  We shall construct the following commutative diagram with admissible exact rows and columns:
\begin{eqnarray*}
\begin{tikzcd}
0 \arrow[r] \arrow[d, "f_{i-1}"] & I \arrow[r, "id_I"] \arrow[d, tail, "f_{k+1}"] & I \arrow[d, tail, dashed, "f"]\\
X_{i-1} \arrow[r, tail, "\alpha_{i-1}^{k-i+2}"] \arrow[d, "id_{X_{i-1}}"] & X_{k+1} \arrow[r, two heads, "p"] \arrow[d, two heads] & coker(\alpha_{i-1}^{k-i+2}) \arrow[d, dashed, two heads]\\
X_{i-1} \arrow[r, tail, "\beta"] & coker(f_{k+1}) \arrow[r, two heads]& coker(\beta)
\end{tikzcd}
\end{eqnarray*}
In the case where $i =1$, we define $X_0 = 0$.  The first two rows and columns are clearly exact.  Since $f_\bullet$ is an admissible monomorphism, $coker(f_\bullet) \in \MMor_k(\E)$, hence $\beta$ is an admissible monomorphism and the third row is exact.

By \cite[Exercise 3.7]{buhler2010exact}, the induced maps forming the third column are uniquely defined and form an admissible short exact sequence.  By injectivity of $I$, $f$ admits a retraction $r: coker(\alpha_{i-1}^{k-i+2}) \twoheadrightarrow I$.  For $1 \le j \le k+1$, define $r_j: X_j \rightarrow I$ to be the composition $r_j = rp\alpha_j^{k+1-j}$.  By the above diagram, $r_j = 0$ for $j \le i-1$; for such $j$ we shall therefore view $r_j$ as a morphism $X_j \rightarrow 0$.  Furthermore, for each $1 \le j < k+1$, $r_{j} = r_{j+1}\alpha_{j}$, hence $r_\bullet: X_\bullet \rightarrow \chi_i(I)_\bullet$ is a morphism in $\MMor_k(\mathcal{E})$.  The verification that $r_\bullet$ is a retraction of $f_\bullet$ is straightforward.  Thus $\chi_i(I)_\bullet$ is injective.

Conversely, suppose $(I_\bullet, \iota_\bullet)$ is injective.  To show each $I_i$ is injective, consider the diagram in $\mathcal{E}$
\begin{eqnarray*}
\begin{tikzcd}
I_i &\\
X \arrow[r, tail, "g"] \arrow[u, "f"] & Y
\end{tikzcd}
\end{eqnarray*}

We must find $h: Y \rightarrow I_i$ making the diagram commute.  Note that $g$ induces an admissible monomorphism $g_\bullet : \chi_i(X)_\bullet \rightarrowtail \chi_i(Y)_\bullet$.  $f$ also induces a morphism $f_\bullet: \chi_i(X)_\bullet \rightarrow I_\bullet$, where $f_j = 0$ for $j < i$, $f_i = f$, and $f_j = X \xrightarrow{f} I_i \rightarrowtail I_{j}$ for $j>i$.  By injectivity of $I_\bullet$, we obtain an induced map $h_\bullet : \chi_i(Y)_\bullet \rightarrow I_\bullet$ such that $f_\bullet = h_\bullet g_\bullet$.  Setting $h = h_i$, we have that $f = hg$, hence $I_i$ is injective.  It follows immediately that the $\iota_i$ are split.

We turn to the classification of the projective objects.  To show that $(P_\bullet, \iota_\bullet)$, with $P_i$ projective and $\iota_i$ split, is projective in $\MMor_k(\mathcal{E})$, it suffices to show that $\chi_i(P)_\bullet$ is projective for any $i$ and any projective $P$.  In fact, something stronger is true; we shall prove that $\chi_i(P)_\bullet$ is projective in $\Mor_k(\E)$.

Let $p_\bullet:  (X_\bullet, f_\bullet) \twoheadrightarrow \chi_i(P)_\bullet$ be an admissible epimorphism in $\Mor_k(\E)$; we shall construct a section $s_\bullet$.  Since $P$ is projective, $p_i: X_i \twoheadrightarrow P$ admits a section $s_i$.  For $j< i$ let $s_j = 0 \rightarrow X_j$, and for $j>i$ let $s_j = P \xrightarrowtail{s_i} X_i \xrightarrow{f_i^{j-i}} X_j$.  It is easy to verify that $s_\bullet: \chi_i(P)_\bullet \rightarrow X_\bullet$ is a morphism in $\MMor_k(\E)$ and a section of $p_\bullet$.  Thus $\chi_i(P)_\bullet$ is projective in $\Mor_k(\E)$, hence also in $\MMor_k(\E)$.

Conversely, let $(P_\bullet, \iota_\bullet)$ be projective in $\MMor_k(\E)$.  To show that $P_i$ is projective, consider the diagram in $\E$
\begin{eqnarray*}
\begin{tikzcd}
& P_i \arrow[d, "f"]\\
Y \arrow[r, two heads, "g"] & X
\end{tikzcd}
\end{eqnarray*}
We must find $h: P_i \rightarrow Y$ making the diagram commute.

We shall define objects $(X_\bullet, \alpha_\bullet), (Y_\bullet, \beta_\bullet) \in \MMor_k(\E)$ and morphisms $f_\bullet: P_\bullet \rightarrow X_\bullet$, $g_\bullet: Y_\bullet \twoheadrightarrow X_\bullet$ such that $X_i = X, Y_i = Y, f_i = f$, and $g_i = g$.  We start by defining $(X_\bullet, \alpha_\bullet)$ and $f_\bullet$.  For all $1 \le j \le i$, let $X_j = X$ and $f_j = f \iota_j^{i-j}$.  For all $1 \le j < i$ let $\alpha_j$ be the identity map on $X$.  For $j \ge i$ we inductively define $X_{j+1}, f_{j+1}$, and $\alpha_j$ via the pushout
\begin{eqnarray*}
\begin{tikzcd}
P_j \arrow[d, "f_j"] \arrow[r, tail, "\iota_{j}"] & P_{j+1} \arrow[d, dashed, "f_{j+1}"]\\
X_j \arrow[r, dashed, tail, "\alpha_j"] & X_{j+1}
\end{tikzcd}
\end{eqnarray*}
Admissible monomorphisms are stable under pushouts, hence $\alpha_i$ is an admissible monomorphism and $f_\bullet: P_\bullet \rightarrow X_\bullet$ is a morphism in $\MMor_k(\E)$.

For $j \le i$, let $Y_j = Y$ and $g_j = g$.  For $j >i$, let $Y_j = Y \oplus X_j$ and $g_j: Y_j \twoheadrightarrow X_j$ be given by $\begin{bmatrix} 0 & id_{X_j} \end{bmatrix}$.  For $j < i$, let $\beta_j = id_Y$.  Let $\beta_i = \begin{bmatrix} id_Y \\ \alpha_i g \end{bmatrix}$ and, for $j > i$, let $\beta_j = \begin{bmatrix} id_Y & 0\\ 0 & \alpha_j \end{bmatrix}$.  The direct sum of admissible monomorphisms is admissible, hence $\beta_j$ is an admissible monomorphism for $j>i$.  $\beta_i$ is an admissible monomorphism by Lemma \ref{direct sum lemma}, therefore $Y_\bullet \in \MMor_k(\E)$.  It is clear that $g_\bullet: Y_\bullet \rightarrow X_\bullet$ is a morphism, that each $g_i$ is an admissible epimorphism, and that $g_\bullet$ has kernel
$$ker(g) \xrightarrowtail{id} \cdots \xrightarrowtail{id}ker(g) \rightarrowtail Y \xrightarrowtail{id} \cdots \xrightarrowtail{id} Y \in \MMor_k(\E)$$
Thus $g_\bullet$ is an admissible epimorphism.

By projectivity of $P_\bullet$, we obtain a morphism $h_\bullet: P_\bullet \rightarrow Y_\bullet$ such that $f_\bullet = g_\bullet h_\bullet$.  Letting $h = h_i$, we have that $f = gh$, hence $P_i$ is projective.

It remains to show that the $\iota_i$ are split.  For any two indices $j > l$, denote $P_j/P_l := coker(\iota^{j-l}_l)$.  It suffices to show that each of the compositions $\begin{tikzcd} P_i \arrow[r, tail, "\iota_i^{k+1-i}"] & P_{k+1} \end{tikzcd}$ is split; this follows immediately if we show that $P_{k+1}/P_i$ is projective for each $1 \le i \le k$.

Suppose we have an admissible epimorphism $g:  Y \twoheadrightarrow X$ and any morphism $f: P_{k+1}/P_i \rightarrow X$; we shall construct a lift $h: P_{k+1}/P_i \rightarrow Y$.  Define $P_\bullet/P_i$ to be the object in $\MMor_k(\E)$ given by $0 \rightarrow \cdots \rightarrow 0 \rightarrow P_{i+1}/P_i \rightarrowtail \cdots \rightarrowtail P_{k+1}/P_i$, with the morphisms induced by the $\iota_j$.  There is a natural morphism $\pi_\bullet: P_\bullet \twoheadrightarrow P_\bullet/P_i$ with kernel
$$P_1 \rightarrowtail \cdots \rightarrowtail P_{i-1} \rightarrowtail P_i \xrightarrowtail{id} \cdots \xrightarrowtail{id} P_i \in \MMor_k(\E)$$
Thus $\pi_\bullet$ is an admissible epimorphism.  Moreover, $f$ and $g$ induce obvious morphisms $f_\bullet: P_\bullet/P_i \rightarrow \chi_{i+1}(X)_\bullet$, and $g_\bullet : \chi_{i+1}(Y)_\bullet \twoheadrightarrow \chi_{i+1}(X)_\bullet$.

Consider the following diagram:
\begin{eqnarray*}
\begin{tikzcd}
& P_{\bullet} \arrow[ddl, bend right, swap, dashed, "h_\bullet"] \arrow[d, two heads, "\pi_{\bullet}"]\\
& P_{\bullet}/P_i \arrow[d, "f_\bullet"] \arrow[dl, dashed, swap, "\overline{h_\bullet}"]\\
\chi_{i+1}(Y)_\bullet \arrow[r, two heads, "g_\bullet"] & \chi_{i+1}(X)_\bullet
\end{tikzcd}
\end{eqnarray*}
By projectivity of $P_\bullet$, we can lift $f_\bullet \pi_\bullet$ to $h_\bullet : P_\bullet \rightarrow \chi_{i+1}(Y)_\bullet$.  Furthermore, since $\chi_{i+1}(Y)_i = 0$, the composition $P_i \rightarrowtail P_{j} \xrightarrow{h_{j}} Y$ is zero for all $j > i$, hence $h_{j}$ factors through $\overline{h_j}: P_{j}/P_i \rightarrow Y$.  Defining $\overline{h_j} = 0$ for $j\le i$, it follows that $h_\bullet = \overline{h_\bullet} \pi_\bullet$, hence $f_\bullet \pi_\bullet  = g_\bullet \overline{h_\bullet} \pi_\bullet$.  Since $\pi_\bullet$ is an epimorphism, we obtain $f_\bullet = g_\bullet \overline{h_\bullet}$, so the above diagram commutes.  In particular, $\overline{h_{k+1}}: P_{k+1}/P_i \rightarrow Y$ is a lift of $f_{k+1} = f$, so $P_{k+1}/P_i$ is projective, as claimed.
\end{proof}

%%%%%%%%%%%%%%%%%%%%%%%%%%%%%%%%%%%

It will also be helpful to have the following characterization of projectives and injectives in $\Mor_k(\E)$.

\begin{proposition}
\label{basic inj/proj characterization} 
Let $\E$ be an exact category.  The object $(P_\bullet, \iota_\bullet) \in \Mor_k(\E)$ is projective if and only if each $P_i$ is projective in $\E$ and each $\iota_i$ is a split monomorphism.  The object $(I_\bullet, \pi_\bullet) \in \Mor_k(\E)$ is injective if and only if each $I_i$ is injective in $\E$ and each $\pi_i$ is a split epimorphism.
\end{proposition}

\begin{proof}
%Take $(P_\bullet, \iota_\bullet) \in \Mor_k(\E)$, with $P_i$ projective and $\iota_i$ split.  Then $P_\bullet \cong \bigoplus_{i=1}^{k+1}\chi_i(P_i')$, where $P_1' = P_1$ and $P_i' = coker(\iota_{i-1})$ for all $i > 1$.  Thus it suffices to prove that $\chi_i(P)_\bullet$ is projective in $\Mor_k(\E)$ for any $i$ and any projective $P$.

%Let $p_\bullet:  (X_\bullet, f_\bullet) \twoheadrightarrow \chi_i(P)_\bullet$ be an admissible epimorphism; we shall construct a section $s_\bullet$.  Since $P$ is projective, $p_i: X_i \twoheadrightarrow P$ admits a section $s_i$.  For $j< i$ let $s_j = 0 \rightarrow X_j$, and for $j>i$ let $s_j = P \xrightarrowtail{s_i} X_i \xrightarrow{f_i^{j-i}} X_j$.  It is easy to verify that $s_\bullet: \chi_i(P)_\bullet \rightarrow X_\bullet$ is a morphism in $\Mor_k(\E)$ and a section of $p_\bullet$.  Thus $\chi_i(P)_\bullet$ is projective.

Let $(P_\bullet, \iota_\bullet)$ be projective in $\Mor_k(\E)$.  To show that $P_i$ is projective, choose any admissible epimorphism $g:Y \twoheadrightarrow X$ in $\E$ and any morphism $f: P_i \rightarrow X$; we must construct $h: P_i \rightarrow Y$ such that $f = gh$.  Define $\omega_i(X)_\bullet \in \Mor_k(\E)$ to be
$$X \xrightarrow{id} \cdots \xrightarrow{id} X \rightarrow 0 \rightarrow \cdots 0$$
where $X$ appears in the first $i$ positions, and similarly for $\omega_i(Y)_\bullet$.  We can extend $f$ to a morphism $f_\bullet :  P_\bullet \rightarrow \omega_i(X)_\bullet$ by setting $f_j := f \iota_j^{i-j}$ for $j \le i$ and $f_j = 0$ for $j > i$; $g$ extends to an admissible epimorphism $g_\bullet: \omega_i(X)_\bullet \twoheadrightarrow \omega_i(Y)_\bullet$ in the obvious way.  By projectivity of $P_\bullet$, we obtain a lift $h_\bullet: P_\bullet \rightarrow Y_\bullet$ such that $f_\bullet = g_\bullet h_\bullet$.  It follows that $f = h_ig$, hence $P_i$ is projective.

To show that $\iota_i$ is a split monomorphism, define
\begin{eqnarray*}
\begin{tikzcd}[column sep = small]
P^{\le i}_\bullet &[-20pt]=&[-20pt] P_1 \arrow[r, "\iota_1"] & \cdots \arrow[r, "\iota_{i-1}"] & P_i \arrow[r] & 0 \arrow[r] & 0 \arrow[r] & \cdots \arrow[r] & 0\\
\widehat{P^{\le i}_\bullet} &=& P_1 \arrow[r, "\iota_1"] & \cdots \arrow[r, "\iota_{i-1}"] & P_i \arrow[r, "id"] & P_i \arrow[r] & 0 \arrow[r] & \cdots \arrow[r] & 0
\end{tikzcd}
\end{eqnarray*}
There are natural morphisms $f_\bullet: P_\bullet \twoheadrightarrow P^{\le i}_\bullet$ and $g_\bullet: \widehat{P^{\le i}_\bullet} \twoheadrightarrow P^{\le i}_\bullet$, both of which are admissible epimorphisms.  By projectivity of $P_\bullet$, we obtain a map $r_\bullet: P_\bullet \rightarrow \widehat{P^{\le i}_\bullet}$ such that $f_\bullet = g_\bullet r_\bullet$.  For all $j \le i$, we have that $f_j = id_{P_j} = g_j$, hence $r_j = id_{P_j}$.  From the diagram
\begin{eqnarray*}
\begin{tikzcd}
P_i \arrow[r, "\iota_i"] \arrow[d, "r_i"] & P_{i+1} \arrow[d, "r_{i+1}"]\\
P_i \arrow[r, "id_{P_i}"] & P_i
\end{tikzcd}
\end{eqnarray*}
we deduce that $r_{i+1}\iota_i = id_{P_i}$, hence $\iota_i$ is a split monomorphism.

For the reverse direction, it suffices to prove that $\chi_i(P)_\bullet$ is projective in $\Mor_k(\E)$ for $1 \le i \le k+1$ and each $P \in \Proj{\E}$.  This claim was proved explicitly in our proof of Proposition \ref{inj/proj characterization}.

Note that there is an equivalence of categories $\Mor_k(\E)^{op} \xrightarrow{\sim} \Mor_k(\E^{op})$ given by $(X_\bullet, f_\bullet) \mapsto (X_{k+2-\bullet}, f_{k+1-\bullet}^{op})$.  The characterization of injective objects thus follows from the characterization of projective objects.
\end{proof}

\begin{remark}
Note that the projective-injective objects of $\MMor_k(\E)$ are precisely the projective objects of $\Mor_k(\E)$.  Dually, the projective-injective objects of $\EMor_k(\E)$ are precisely the injective objects of $\Mor_k(\E)$.
\end{remark}
%%%%%%%%%%%%%%%%%%%%%%%%%%%

If an exact category has enough injectives or projectives, so does its monomorphism category.

\begin{proposition}
\label{enough inj/proj}
Let $\E$ be an exact category.  If $\E$ has enough projectives (resp., injectives), then so does $\MMor_k(\E)$.
\end{proposition}

\begin{proof}
Let $(X_\bullet, \alpha_\bullet) \in \MMor_k(\E)$, and suppose $\E$ has enough projectives.  Then there exist projective objects $P_i$ and admissible epimorphisms $p_i: P_i \twoheadrightarrow X_i$ for each $1 \le i \le k+1$.  Let $P'_i = \bigoplus_{j=1}^i P_j = P'_{i-1}\oplus P_i$ and let $\iota_i:  P'_i \rightarrowtail P'_{i+1}$ denote the canonical monomorphism.  Then $(P'_\bullet, \iota_\bullet)$ is projective in $\MMor_k(\E)$ by Proposition \ref{inj/proj characterization}.  Define $f_\bullet: P'_\bullet \rightarrow X_\bullet$ by $f_i := \begin{bmatrix} \alpha_1^{i-1}p_1 & \cdots & \alpha_{i-1}p_{i-1} & p_i \end{bmatrix} = \begin{bmatrix}  \alpha_{i-1}f_{i-1} & p_i \end{bmatrix}$.  Since $p_i$ is an admissible epimorphism in $\E$, by Lemma \ref{direct sum lemma} so is $f_i$, hence $f_\bullet$ is an admissible epimorphism in $\Mor_k(\E)$.  Let $g_\bullet: (K_\bullet, \beta_\bullet) \rightarrowtail (P'_\bullet, \iota_\bullet)$ be the kernel of $f_\bullet$.  To show that $f_\bullet$ is admissible in $\MMor_k(\E)$, we must show that $(K_\bullet, \beta_\bullet)$ lies in $\MMor_k(\E)$.

Write the admissible monomorphism $g_i: K_i \rightarrowtail P'_i = P'_{i-1}\oplus P_i$ as $g_i = \begin{bmatrix} \psi_i \\ -\varphi_i \end{bmatrix}$.  We have an admissible short exact sequence
\begin{eqnarray*}
\begin{tikzcd}[ampersand replacement=\&]
K_i \arrow[r, tail, "\begin{bmatrix} \psi_i \\ -\varphi_i \end{bmatrix}"] \& P'_{i-1} \oplus P_i \arrow[two heads]{rr}{\begin{bmatrix} \alpha_{i-1}f_{i-1} & p_i \end{bmatrix}} \& \& X_i
\end{tikzcd}
\end{eqnarray*}
which gives rise to the bicartesian square:
\begin{eqnarray*}
\begin{tikzcd}
 K_i \arrow[d, "\varphi_i"] \arrow[r, "\psi_i", two heads] & P'_{i-1} \arrow[d, "\alpha_{i-1} f_{i-1}"]\\
 P_i \arrow[r, two heads, "p_i"] & X_i
\end{tikzcd}
\end{eqnarray*}

Since $p_i$ is an admissible epimorphism, so is $\psi_i$.  By projectivity of $P'_{i-1}$, the top row is split exact, hence $K_i \cong P'_{i-1} \oplus ker(\psi_i)$.  Identifying the two, we can express $\psi_i$ as $\begin{bmatrix} id & 0 \end{bmatrix}$ and $\varphi_i$ as $\begin{bmatrix} \tau_i & \theta_i \end{bmatrix}$ for some $\tau_i : P'_{i-1} \rightarrow P_i$ and $\theta_i: ker(\psi_i) \rightarrow P_i$.  In particular, we can express $g_i: K_i \rightarrow P'_i$ as the matrix $\begin{bmatrix} id & 0 \\ -\tau_i & -\theta_i \end{bmatrix}$.

Let us express $\beta_{i-1}: K_{i-1} \rightarrow K_i = P'_{i-1} \oplus ker(\psi_i)$ as $\begin{bmatrix} \delta_{i-1} \\ \gamma_{i-1} \end{bmatrix}$.  We can then rewrite the identity $g_i\beta_{i-1} = \iota_{i-1}g_{i-1}$ as the commutative diagram
\begin{eqnarray*}
\begin{tikzcd}[ampersand replacement=\&]
K_{i-1} \arrow[r, "\begin{bmatrix} \delta_{i-1} \\ \gamma_{i-1} \end{bmatrix}"]  \arrow[d, tail, "g_{i-1}"] \& P'_{i-1} \oplus ker(\psi_i) \arrow[tail]{d}{\begin{bmatrix} id & 0 \\ -\tau_i & -\theta_i \end{bmatrix}} \\
P'_{i-1} \arrow[r, tail, "\begin{bmatrix} id \\ 0 \end{bmatrix}"]\& P'_{i-1}\oplus P_i
\end{tikzcd}
\end{eqnarray*}
It follows that $\delta_{i-1} = g_{i-1}$.  Since $g_{i-1}$ is an admissible monomorphism, so is $\beta_{i-1} = \begin{bmatrix} g_{i-1} \\ \gamma_{i-1} \end{bmatrix}$.  Thus $(K_\bullet, \beta_\bullet) \in \MMor_k(\E)$, and so $f_\bullet$ is an admissible epimorphism.  Therefore $\MMor_k(\E)$ has enough projectives.

Suppose now that $\E$ has enough injectives.  Let $(X_\bullet, \alpha_\bullet) \in \MMor_k(\E)$; we shall construct an admissible monomorphism $g_\bullet:  (X_\bullet, \alpha_\bullet) \rightarrowtail (I_\bullet, \iota_\bullet)$ for some injective object $(I_\bullet, \iota_\bullet)$.

Let $g_1: X_1 \rightarrowtail I_1$ be an admissible morphism from $X_1$ to an injective object in $I_1 \in \E$; we shall define the remaining admissible monomorphsims $g_i$, injective objects $I_i$, and split monomorphisms $\iota_i$ inductively.  Suppose we have constructed $g_i: X_i \rightarrowtail I_i$.  Since $\alpha_i: X_i \rightarrowtail X_{i+1}$ is an admissible monomorphism, we can lift $g_i$ to a morphism $\hat{g_i}: X_{i+1} \rightarrow I_i$.  Since $\E$ has enough injectives, there exists an admissible monomorphism $h_{i+1}: coker(\alpha_i) \rightarrowtail I'_{i+1}$ for some injective object $I'_{i+1}$.  We define $I_{i+1} := I_i \oplus I'_{i+1}$ and $g_{i+1} = \begin{bmatrix} \hat{g_i} & h_{i+1}\pi_{i+1} \end{bmatrix}$, where $\pi_{i+1}: X_{i+1} \twoheadrightarrow coker(\alpha_i)$ is the canonical map.  Let $\iota_i: I_i \rightarrowtail I_{i+1}$ be the inclusion of $I_i$ as a direct summand of $I_{i+1}$.  Since $I_i$ and $I'_{i+1}$ are injective, so is $I_{i+1}$.  It is clear that $\iota_i$ is split; it remains to check that $g_{i+1}$ is an admissible monomorphism.

We have a commutative diagram with exact rows
\begin{eqnarray*}
\begin{tikzcd}
X_i \arrow[r, tail, "\alpha_i"] \arrow[d, tail, "g_i"] & X_{i+1} \arrow[r, two heads, "\pi_{i+1}"] \arrow[d, "g_{i+1}"] & coker(\alpha_i) \arrow[d, tail, "h_{i+1}"]\\
 I_i \arrow[r, tail, "\iota_{i}"] & I_{i+1} \arrow[r, two heads] & I'_{i+1}
 \end{tikzcd}
\end{eqnarray*}
It follows from the Five Lemma \cite[Corollary 3.2]{buhler2010exact} that $g_{i+1}$ is an admissible monomorphism, hence $g_\bullet$, $I_\bullet$, and $\iota_\bullet$ are defined, and $g_\bullet$ is an admissible morphism in $\Mor_k(\E)$.

To see that $g_\bullet$ is an admissible monomorphism in $\MMor_k(\E)$, we must show that its cokernel $(Q_\bullet, \psi_\bullet)$ lies in $\MMor_k(\E)$.  We have a commutative diagram with exact columns:
\begin{eqnarray*}
\begin{tikzcd}
X_i \arrow[r, tail, "\alpha_i"] \arrow[d, tail, "g_i"] & X_{i+1} \arrow[r, two heads] \arrow[d, tail, "g_{i+1}"] & coker(\alpha_i) \arrow[d, tail, "h_{i+1}"]\\
I_i \arrow[r, tail, "\iota_i"] \arrow[d, two heads] & I_{i+1} \arrow[r, two heads] \arrow[d, two heads] & I'_i \arrow[d, two heads]\\
coker(g_i) \arrow[r, dashed, "\psi_i"] & coker(g_{i+1}) \arrow[r, dashed] & coker(h_{i+1})
\end{tikzcd}
\end{eqnarray*}
Since the first two rows are exact, by the $3\times 3$ Lemma \cite[Corollary 3.6]{buhler2010exact} the third row is also an admissible short exact sequence.  In particular, $\psi_i$ is an admissible monomorphism, hence $coker(g_\bullet) \in \MMor_k(\E)$.  Thus $g_\bullet$ is an admissible monomorphism.  $I_\bullet$ is injective by Proposition \ref{inj/proj characterization}, hence $\MMor_k(\E)$ has enough injectives.
\end{proof}

We have arrived at the main result of this section:

\begin{theorem}
\label{frobenius}
Let $\E$ be a Frobenius exact category.  Then $\MMor_k(\E)$ is Frobenius exact.
\end{theorem}

\begin{proof}
Since $\Proj{\E} = \Inj{\E}$, it follows immediately from Proposition \ref{inj/proj characterization} that $\Proj{\MMor_k(\E)} = \Inj{\MMor_k(\E)}$.  Since $\E$ has enough projectives and injectives, by Proposition \ref{enough inj/proj} so does $\MMor_k(\E)$.  
\end{proof}

%%%%%%%%%%%%%%%%%%%%%%%%%%%

\begin{definition}
Let $\E$ be a Frobenius exact category.  For $N \ge 2$, define the \textbf{$N$-stable category of $\E$}, denoted $\stab{N}{\E}$, to be the stable category of $\MMor_{N-2}(\E)$.
\end{definition}

Note that when $N=2$, we obtain the stable category of $\E$.

%%%%%%%%%%%%%%%%%%%%%%%%%%%%%%%%%%%%%%%%%%%%%%%%%%%%%%%%%%%%%%%%%%%%%%%%%%%%%%%%%%%%%%%%%%%%%%%%%%%%%%%%%%%%%%%%%%%%%%%%%%%%%%%%%%%%%%%%%%%%%%%%%%%%%%%%%%%%%%%%%%

\section{Acyclic Projective-Injective $N$-Complexes}
\label{Acyclic Projective-Injective N-Complexes}

Throughout this section, let $\mathcal{A}$ denote an abelian category which is Frobenius exact.  Consider the functor $F: C^{ac}_N(\Proj{\mathcal{A}}) \rightarrow \MMor_{N-2}(\mathcal{A})$ given by
\begin{align*}
F(P^\bullet) &= Z^{0}_1(P^\bullet) \hookrightarrow  \cdots \hookrightarrow Z^{0}_{N-1}(P^\bullet)
\end{align*}

In this section, we shall prove that $F$ induces an equivalence $\overline{F}$ between $K^{ac}_N(\Proj{\mathcal{A}})$ and $\stab{N}{\mathcal{A}}$.

%%%%%%%%%%%%%%%%%%%%%%%%%%%
\subsection{Properties of $F$}

To prove that $F$ is full, we introduce the following terminology.

\begin{definition}
Let $P^\bullet, Q^\bullet \in C_N(\A)$.  Let $n \in \Z$ and let $f^n: P^n \rightarrow Q^n$ be any morphism.  We say $f^n$ \textbf{preserves cycles} if the restriction of $f^n$ to $Z^n_i(P^\bullet)$ has image in $Z^n_i(Q^\bullet)$ for each $1 \le i \le N-1$.

Similarly, we say $f^n$ \textbf{preserves boundaries} if the restriction of $f^n$ to $B^n_i(P^\bullet)$ has image in $B^n_i(Q^\bullet)$ for each $1 \le i \le N-1$.
\end{definition}

Note that if $P^\bullet$ and $Q^\bullet$ are both acyclic, then the two notions are equivalent.

\begin{proposition}
\label{fullness on complexes}
$F$ is full.
\end{proposition}
\begin{proof}
Take $P^\bullet, Q^\bullet \in C_N^{ac}(\Proj{A})$ and $f_\bullet: F(P^\bullet) \rightarrow F(Q^\bullet)$.  Using the injectivity of $Q^{0}$, lift the map $Z^{0}_{N-1}(P^\bullet) \xrightarrow{f_{N-1}} Z^{0}_{N-1}(Q^\bullet) \hookrightarrow Q^{0}$ along the monomorphism $Z^0_{N-1}(P^\bullet) \hookrightarrow P^0$ to obtain a morphism $f^{0}: P^{0} \rightarrow Q^{0}$.  Clearly, the restriction of $f^{0}$ to $Z^{0}_i(P^\bullet)$ is $f_i$, hence $f^0$ preserves cycles.

It thus suffices to extend $f^0$ to a morphism of complexes $f^\bullet: P^\bullet \rightarrow Q^\bullet$.  We claim that, given a morphism $f^n: P^n \rightarrow Q^n$ which preserves cycles, we can construct maps $f^{n\pm 1}: P^{n\pm 1} \rightarrow Q^{n\pm 1}$, both preserving cycles, such that $d_Q^i f^i = f^{i+1}d_P^i$ for $i = n-1, n$.  Once this claim established, we can extend $f^0$ to $f^\bullet$ by induction, proving fullness.

Since $f^n$ preserves cycles, we obtain an induced map on the images $\overline{f^n}: B^{n+1}_{N-1}(P^\bullet) \rightarrow B^{n+1}_{N-1}(Q^\bullet)$, which, by injectivity of $Q^{n+1}$, lifts to a map $f^{n+1}: P^{n+1} \rightarrow Q^{n+1}$.  It follows immediately that $f^{n+1}d_P^n = d_Q^n f^n$.  Precomposing with $d_P^{n+1-i, i}$, and using the fact that $f^n$ preserves boundaries, we deduce that $f^{n+1}$ preserves boundaries and therefore cycles.

Since $f^n$ preserves boundaries, it restricts to a map from $B^n_{N-1}(P^\bullet)$ to $B^n_{N-1}(Q^\bullet)$.  Using projectivity of $P^{n-1}$, we can lift this restriction to $f^{n-1}: P^{n-1} \rightarrow Q^{n-1}$.  It follows immediately that $f^{n}d_P^{n-1} = d_Q^{n-1} f^{n-1}$, hence $f$ maps $Z^{n-1}_1(P^\bullet)$ into $Z^{n-1}_1(Q^\bullet)$.  Postcomposing this equation with $d_Q^{n, i-1}$, and using induction on $i$, we deduce that $f^n$ maps $Z^{n-1}_i(P^\bullet)$ into $Z^{n-1}_i(Q^\bullet)$, hence $f^{n-1}$ preserves cycles.
\end{proof}

%%%%%%%%%%%%%%%%%%%%%%%%%%%

To show that $F$ is essentially surjective, it will be convenient to introduce the following terminology.

\begin{definition}
\label{N-acyclic array}
An \textbf{$N$-acyclic array} in $\A$ is the data of:\\
$\bullet$ objects $X^n_j$; $n \in \Z, 0 \le j \le N$\\
$\bullet$ monomorphisms $\iota^n_j: X^n_j \hookrightarrow X^n_{j+1}$; $n \in \Z, 0 \le j < N$\\
$\bullet$ epimorphisms $p^n_j: X^n_{j} \twoheadrightarrow X^{n+1}_{j-1}$; $n \in \Z, 0 < j \le N$\\
We shall write $\iota^{n, k}_j: X^n_j \hookrightarrow X^n_{j+k}$ for the composition $\iota^n_{j+k-1} \cdots \iota^n_j$ of $k$ successive $\iota^n_\bullet$, beginning at $\iota^n_j$, and similarly for $p^{n,k}_j: X^n_j \twoheadrightarrow X^{n+k}_{j-k}$.

The above data should satisfy the following three properties:\\
1)  $X^n_0 \cong 0$.\\
2)  $X^n_N$ is projective-injective.\\
3)  For all $1 \le j \le N-1$, the diagram
\begin{eqnarray*}
\begin{tikzcd}
& X^n_{j+1} \arrow[dr, two heads, "p^n_{j+1}"]&\\
X^n_j \arrow[ur, hook, "\iota^n_j"] \arrow[dr, two heads, "p^n_j"] & & X^{n+1}_j\\
& X^{n+1}_{j-1} \arrow[ur, hook, "\iota^{n+1}_{j-1}"] &
\end{tikzcd}
\end{eqnarray*}
commutes and forms a bicartesian square.

Given $X_\bullet \in \MMor_{N-2}(\A)$, we say that the $N$-acyclic array $(X^n_j, \iota^n_j, p^n_j)$ \textbf{extends} $X_\bullet$ if $X_\bullet = (X^0_\bullet, \iota^0_\bullet)$. 
\end{definition}

Given $P^\bullet \in C^{ac}_N(\Proj{\A})$, it is easily verified that we obtain an $N$-cyclic array by defining $X^n_j = Z^n_j(P^\bullet)$ (here we take $Z^n_0(P^\bullet) = 0$ and $Z^n_N(P^\bullet) = P^n$), $\iota^n_j$ to be the inclusion of kernels, and $p^n_j$ to be the morphism on kernels induced by $d^n_P$.

%%%%%%%%%%%%%%%%%%%%%%%%%%%

\begin{proposition}
\label{essential surjectivity on complexes}
$F$ is essentially surjective.
\end{proposition}

\begin{proof}
Let $(X_\bullet, \iota_\bullet) \in \MMor_k(\mathcal{A})$.  The proof proceeds in two steps.  First we prove that, given an $N$-acyclic array $(X^n_j, \iota^n_j, p^n_j)$ extending $X_\bullet$, there exists $P^\bullet \in C^{ac}_N(\Proj{\A})$ such that $F(P^\bullet) = X_\bullet$.  In the second step, we shall construct such an $N$-acyclic array.

Given an $N$-acyclic array $(X^n_j, \iota^n_j, p^n_j)$ extending $X_\bullet$, define maps
$$d^n := \iota^{n+1}_{N-1}p^n_N: X^n_N \rightarrow X^{n+1}_N$$
We claim that $(X^\bullet_N, d^\bullet) \in C^{ac}_N(\Proj{\A})$.  By assumption, all $p$ and $\iota$ commute, so we have that $d^{n, j} = \iota^{n+j, j}_{N-j} p^{n, j}_N$ for all $1 \le j \le N$.  In particular, $d^{n,N}$ factors through $X^{n+N}_0 = 0$, hence $X^\bullet_N \in C_N(\A)$.  Each $X^n_N$ is projective-injective by assumption.

To show that $X^\bullet_N$ is acyclic, note that
\begin{eqnarray*}
 Z^n_j(X^\bullet_N) &=& ker(d^{n,j}) = ker(\iota^{n+j, j}_{N-j} p^{n, j}_N)\\
 &= & ker(p^{n,j}_N)\\
B^n_j(X^\bullet_N) &=& im(d^{n-N+j,N-j}) = im(\iota^{n, N-j}_{j} p^{n-N+j, N-j}_N)\\
 &=& X^n_j
\end{eqnarray*}
Thus we must show that $ X^n_j = ker(p^{n,j}_N)$.  Since the composition of bicartesian squares is bicartesian, the commutative square
\begin{eqnarray*}
\begin{tikzcd}
& X^n_{j+k} \arrow[dr, two heads, "p^{n, j}_{j+k}"]&\\
X^n_j \arrow[ur, hook, "\iota^{n,k}_j"] \arrow[dr, two heads, "p^{n,j}_j"] & & X^{n+j}_{k}\\
& X^{n+j}_{0} = 0 \arrow[ur, hook, "\iota^{n+j, k}_{0}"] &
\end{tikzcd}
\end{eqnarray*}
is bicartesian for all $1 \le j \le N-1, 1 \le k \le N-j$.  This yields an exact sequence
\begin{eqnarray*}
\begin{tikzcd}
0 \arrow[r] & X^n_j \arrow[r, hook, "\iota^{n,k}_j"] & X^n_{j+k} \arrow[r, two heads, "p^{n, j}_{j+k}"] & X^{n+j}_{k} \arrow[r] & 0
\end{tikzcd}
\end{eqnarray*}
Taking $k = N-j$, we obtain that $X^n_j = ker(p^{n,j}_N)$, as desired.  Therefore $X^\bullet_N$ is acyclic.

Taking $n = 0$ and $k=1$ in the above exact sequence, we see that the morphism $Z^0_j(X^\bullet_N) \hookrightarrow Z^0_{j+1}(X^\bullet_N)$ is precisely $X^0_j \xhookrightarrow{\iota^0_j} X^0_{j+1}$.  Thus $F(X^\bullet_N) = X_\bullet$.  Thus $P^\bullet := X^\bullet_N$ satisfies the desired properties.

We must now construct an $N$-acyclic array extending $(X_\bullet, \iota_\bullet)$.  For $1 \le j \le N-1$, let $X^0_j = X_j$ and let $X^0_0 = 0$.  For $1 \le j \le N-2$, let $\iota^0_j = \iota_j$ and let $\iota^0_0: 0 \hookrightarrow X_1$ be the zero map. Define $\iota^0_{N-1}: X^0_{N-1} \hookrightarrow X^0_{N}$ to be the inclusion of $X^0_{N-1}$ into a projective-injective object $X^0_{N-1}$.

Suppose for some $n\ge 0$ we have constructed, for all $j$, $X^n_j$ and $\iota^n_j$.  Define $X^{n+1}_0 = 0$ and $p^n_1: X^n_1 \twoheadrightarrow 0$.  Next, inductively define $X^{n+1}_j$, $i^{n+1}_{j-1}$, and $p^{n}_{j+1}$ for $1 \le j \le N-1$ via iterated pushouts
\begin{eqnarray}
\label{bicartesian}
\begin{tikzcd}
& X^n_{j+1} \arrow[dr, two heads, dashed, "p^n_{j+1}"] & \\
X^n_j \arrow[ur, hook, "\iota^n_j"] \arrow[dr, two heads, "p^n_j"]  & & X^{n+1}_{j}\\
& X^{n+1}_{j-1} \arrow[ur, hook, dashed, "\iota^{n+1}_j"] &
\end{tikzcd}
\end{eqnarray}
In abelian categories, pushouts preserve both epimorphisms and monomorphisms, hence the newly defined maps $p$ and $\iota$ are also epimorphisms and monomorphisms, respectively.  Finally, define $\iota^{n+1}_{N-1}: X^{n+1}_{N-1} \hookrightarrow X^{n+1}_N$ to be an inclusion of $X^{n+1}_{N-1}$ into a projective-injective object $X^{n+1}_N$.  Note that we have now constructed $X^{n+1}_j$, $\iota^{n+1}_j$, and $p^n_j$ for all $j$.  Proceeding inductively, we can define $X^n_j$, $\iota^n_j$, and $p^n_j$ for all $n \ge 0$ and for all $j$.

For $n \le 0$, the construction is dual.  Having defined $X^n_j$ and $\iota^n_j$ for all $j$, define $p^{n-1}_N: X^{n-1}_N \twoheadrightarrow X^n_{N-1}$ to be a surjection from a projective-injective object $X^{n-1}_N$.  Then $X^{n-1}_j, i^{n-1}_j,$ and $p^{n-1}_j$ are defined via iterated pullbacks for $N-1 \ge j \ge 1$.  Finally, define $X^{n-1}_0  =0$ and $\iota^{n-1}_0$ to be the zero map.

It is immediate that $(X^n_j, \iota^n_j, p^n_j)$ satisfies properties 1 and 2 of Definition \ref{N-acyclic array}.  To see that property 3 holds, note that each commutative square in (\ref{bicartesian}) is, by construction, either a pullback ($n < 0$) or pushout ($n\ge 0$).  But since the $\iota$ are monomorphisms and the $p$ are epimorphisms, any such pullback or pushout square is automatically bicartesian.  Thus the data we have constructed form an $N$-acyclic array which extends $(X_\bullet, \iota_\bullet)$.
\end{proof}

The category $C^{ac}_N(\Proj{\A})$ inherits the structure of an exact category from $C_N(\A)$.

%%%%%%%%%%%%%%%%%%%%%%%%%%%

\begin{proposition}
$C^{ac}_N(\Proj{\A})$ is a fully exact subcategory of $C_N(\A)$.  An object $P^\bullet \in C^{ac}_N(\Proj{\A})$ is projective (resp., injective) if and only if it is projective (resp., injective) in $C_N(\A)$.  Thus $C^{ac}_N(\Proj{\A})$ is Frobenius exact.
\end{proposition}

\begin{proof}
$C^{ac}_N(\Proj{\A})$ is clearly a full, additive subcategory of $C_N(\A)$.  Given a chainwise-split short exact sequence $X^\bullet \rightarrowtail Y^\bullet \twoheadrightarrow Z^\bullet$ with $X^\bullet, Z^\bullet \in C^{ac}_N(\Proj{\A})$ and $Y^\bullet \in C_N(\A)$, it is clear that $Y^n \in \Proj{\A}$ for all $n \in \Z$.  Since $X^\bullet$ and $Z^\bullet$ are acyclic, it follows immediately from the long exact sequence in homology that $Y^\bullet$ is acyclic.  Thus $C^{ac}_N(\Proj{A})$, together with the class of all chainwise split exact sequences, is a fully exact subcategory of $C_N(\A)$.  The proof of \cite[Theorem 2.1]{iyama2017derived} applies without change to $C^{ac}_N(\Proj{\A})$, hence the projective and injective objects are direct sums of complexes of the form $\mu^n_N(P)$, where $P \in \Proj{\A}$.  The second and third statements follow immediately.
\end{proof}

%%%%%%%%%%%%%%%%%%%%%%%%%

\begin{proposition}
$F: C^{ac}_N(\Proj{\A}) \rightarrow \MMor_{N-2}(\A)$ preserves short exact sequences.
\end{proposition}

\begin{proof}
Consider a chainwise split exact sequence $\begin{tikzcd} P^\bullet \arrow[r, tail, "f^\bullet"] & Q^\bullet \arrow[r, two heads, "g^\bullet"] & R^\bullet \end{tikzcd}$ in $C^{ac}_N(\Proj{\A})$.  Applying the Snake Lemma to
\begin{eqnarray*}
\begin{tikzcd}
0 \arrow[r] & P^{0} \arrow[d, "d_P^{0, j}"] \arrow[r, hook, "f^0"] & Q^{0} \arrow[d, "d_Q^{0,j}"] \arrow[r, two heads, "g^0"] & R^{0} \arrow[d, "d_R^{0,j}"] \arrow[r] & 0\\
0 \arrow[r] & P^{j} \arrow[r, hook, "f^j"] & Q^{j} \arrow[r, two heads, "g^j"] & R^{j} \arrow[r] & 0
\end{tikzcd}
\end{eqnarray*}
we obtain an exact sequence
$$0 \rightarrow Z^{0}_j(P^\bullet) \hookrightarrow Z^{0}_j(Q^\bullet) \rightarrow Z^{0}_j(R^\bullet) \xrightarrow{\phi} coker(d_P^{0,j})$$
It remains to show that the connecting morphism $\phi$ is zero.

We briefly recall the construction of $\phi$.  Let $X$ be the pullback
\begin{eqnarray*}
\begin{tikzcd}
X \arrow[d, hook, "\iota"] \arrow[r, two heads, "p"] & Z^{0}_j(R^\bullet) \arrow[d, hook]\\
Q^{0} \arrow[r, two heads, "g^0"] & R^{0}
\end{tikzcd}
\end{eqnarray*}
From this diagram we see that $g^j\circ d_Q^{0,j} \iota = 0$, hence $d_Q^{0,j} \iota$ factors through $ker(g^j) = f^j$.  Write $d_Q^{0,j} \iota$ as $X \xrightarrow{\alpha} P^{j} \xhookrightarrow{f^j} Q^{j}$ for a unique map $\alpha$.  Then $\phi$ is given by the induced map on cokernels
\begin{eqnarray*}
\begin{tikzcd}
0 \arrow[r] & ker(p) \arrow[d, "\sim"] \arrow[r, hook] & X \arrow[d, "\alpha"] \arrow[r, two heads, "p"] & Z^{0}_j(R^\bullet) \arrow[d, dashed, "\phi"] \arrow[r] & 0\\
 & P^{0} \arrow[r, "d_P^{0,j}"] & P^{j} \arrow[r, two heads] & coker(d_P^{0,j}) \arrow[r] & 0
\end{tikzcd}
\end{eqnarray*}
Thus for $\phi$ to be zero, we must show that $\alpha$ factors through $im(d_P^{0,j})$.

Since $P^\bullet \rightarrowtail Q^\bullet \twoheadrightarrow R^\bullet$ is chainwise split exact, for each $n$ we can write $Q^n \cong P^n \oplus R^n$, with $f^n$ and $g^n$ becoming the canonical inclusion and projection maps, respectively.  Using this decomposition, we can express
\begin{eqnarray*}
\iota &=& \begin{bmatrix}\iota_1 \\ \iota_2 \end{bmatrix}\\
d_Q^{0,j} &=& \begin{bmatrix} d_P^{0,j} & \beta\\ 0 & d_R^{0,j} \end{bmatrix}\\
d_Q^{j, N-j} &=& \begin{bmatrix} d_P^{j,N-j} & \gamma\\ 0 & d_R^{j, N-j} \end{bmatrix}
\end{eqnarray*}
Note that $d^{0,j}_R\iota_2 = d^{0,j}_R g^0 \iota = d^{0,j}_R p = 0$.  It follows that
\begin{eqnarray*}
d_Q^{0,j} \iota = \begin{bmatrix} d_P^{0,j} & \beta\\ 0 & d_R^{0,j} \end{bmatrix} \begin{bmatrix}\iota_1 \\ \iota_2 \end{bmatrix} = \begin{bmatrix} d_P^{0,j} \iota_1 + \beta \iota_2 \\ 0 \end{bmatrix}\\
\end{eqnarray*}
hence $\alpha = d_P^{0,j} \iota_1 + \beta \iota_2$.  Furthermore,
\begin{eqnarray*}
0 = d_Q^{j, N-j} \circ d_Q^{0,j}\iota = \begin{bmatrix} d_P^{j,N-j} & \gamma\\ 0 & d_R^{j, N-j} \end{bmatrix} \begin{bmatrix} d_P^{0,j} \iota_1 + \beta \iota_2 \\ 0 \end{bmatrix} = \begin{bmatrix} d_P^{j, N-j}\beta \iota_2 \\ 0 \end{bmatrix}
\end{eqnarray*}
We have that $\beta \iota_2$ factors through $Z^{j}_{N-j}(P^\bullet) = im(d_P^{0,j})$, hence so does $\alpha = d_P^{0,j} \iota_1 + \beta \iota_2$.  Thus $\phi = 0$ and so $0 \rightarrow Z^{0}_j(P^\bullet) \rightarrow Z^{0}_j(Q^\bullet) \rightarrow Z^{0}_j(R^\bullet) \rightarrow 0$ is exact for each $j$.
\end{proof}

%%%%%%%%%%%%%%%%%%%%%%%%%%%

\begin{corollary}
$F$ descends to a functor $\overline{F}: K_N^{ac}(\Proj{\mathcal{A}}) \rightarrow \stab{N}{\mathcal{A}}$ of triangulated categories.
\end{corollary}

\begin{proof}
By Proposition \ref{inj/proj characterization}, for any $i \in \Z$, $F(\mu^i_N(P))$ is projective-injective in $\MMor_{N-2}(\A)$.  Thus $F$ preserves projective-injective objects and so descends to a functor $\overline{F}$ between the stable categories.  Since $F$ preserves exact sequences and projective-injective objects, it follows immediately that $\overline{F}$ preserves distinguished triangles and the suspension functor, hence is a functor of triangulated categories.
\end{proof}

%%%%%%%%%%%%%%%%%%%%%%%%%%%%%%%%%%%%%%%%%%%%%%%%%%%%%%%%%%%%%%%%%%%%%%%%%%%%%%

\subsection{Properties of $\overline{F}$}

In this section, we shall prove that $\overline{F}$ is an equivalence of categories.  Most of our work will be to show that $\overline{F}$ is faithful.  The following terminology will be convenient for the proof.

\begin{definition}
Let $f^\bullet: P^\bullet \rightarrow Q^\bullet$ be a morphism in $K_N^{ac}(\Proj{\A})$.  Given a family of morphisms $h^i:  P^i \rightarrow Q^{i-N+1}$, we define the sum
$$S_h(n, j, k) := \sum_{i=n+j}^{n+k-1} d_Q^{\circ, n-i+N-1} h^i d_P^{n, i-n}: P^n \rightarrow Q^n$$
whenever the $h^i$ appearing in the formula are defined.  To understand this expression, note that $f^\bullet$ is null-homotopic if and only if $h^i$ is defined for all $i \in \Z$ and $f^n= S_h(n, 0, N)$ for each $n \in \Z$.  Increasing the second parameter removes terms from the start of the sum, and decreasing the third parameter removes terms from the end of the sum.

We define a \textbf{homotopy (of $f^\bullet$) at $n$} to be a sequence of $N$ maps $(h^n, h^{n+1}, \cdots, h^{n+N-1})$ such that $f^n = S_h(n, 0, N)$.  We define a \textbf{seed (of $f^\bullet$) at $n$} to be a sequence of $N-1$ maps $(h^n, h^{n+1}, \cdots, h^{n+N-2})$ such that $f^n|_{Z^n_{N-1}(P^\bullet)} = S_h(n, 0, N-1)|_{Z^n_{N-1}(P^\bullet)}$.
\end{definition}

%%%%%%%%%%%%%%%%%%%%%%%%%%%

The following lemma is trivial when $N=2$.

\begin{lemma}
\label{seed lemma 1}
Let $f^\bullet: P^\bullet \rightarrow Q^\bullet$ be a morphism in $K_N^{ac}(\Proj{\A})$. If $\overline{F}(f) = 0$, then there exists a seed of $f^\bullet$ at $0$.
\end{lemma}

\begin{proof}
Since $\overline{F}(f) = 0$, we have a diagram
\begin{eqnarray*}
\begin{tikzcd}
Z^{0}_1(P^\bullet) \arrow[r, hook] \arrow[d] & Z^{0}_2(P^\bullet) \arrow[r, hook] \arrow[d] & \cdots \arrow[r, hook] & Z^{0}_{N-1}(P^\bullet) \arrow[d]\\
I_1 \arrow[r, hook] \arrow[d] & I_1 \oplus I_2 \arrow[r, hook] \arrow[d] & \cdots \arrow[r, hook] & \bigoplus_{j=1}^{N-1}I_{j} \arrow[d]\\
Z^{0}_1(Q^\bullet) \arrow[r, hook] & Z^{0}_2(Q^\bullet) \arrow[r, hook] & \cdots \arrow[r, hook] & Z^{0}_{N-1}(Q^\bullet)
\end{tikzcd}
\end{eqnarray*}
where the horizontal maps are canonical inclusions, the $I_j$ are projective-injective, and the $j$th pair of vertical maps composes to $f^{0}|_{Z^{0}_j(P^\bullet)}$.  For $1 \le j \le N-1$, let $a_j:  Z^{0}_{N-1}(P^\bullet) \rightarrow I_j$ and $b_j: I_j \rightarrow Z^{0}_{N-1}(Q^\bullet)$ denote the components of the rightmost vertical maps, so that we have $f^{0}|_{Z^{0}_{N-1}(P^\bullet)} = \sum_{j=1}^{N-1} b_ja_j$.

For each $1 \le i \le N-1$, by commutativity of the top rows we have that $a_{i}$ factors through $Z^{0}_{N-1}(P^\bullet)/Z^{0}_{i-1}(P^\bullet)$. (For the degenerate case $i=1$ we let $Z^0_0(P^\bullet) = 0$.)  By injectivity of $I_{i}$, we obtain a commutative diagram 
\begin{eqnarray*}
\begin{tikzcd}
Z^{0}_{N-1}(P^\bullet) \arrow[dr, "a_{i}", swap] \arrow[r, two heads] & Z^{0}_{N-1}(P^\bullet)/Z^{0}_{i-1}(P^\bullet) \arrow[d, "\overline{a_{i}}"] \arrow[r, hook, "\overline{d_P^{0, i-1}}"] & P^{i-1} \arrow[dl, dashed, "\alpha^{i-1}"] \\
& I_{i} & 
\end{tikzcd}
\end{eqnarray*}
Thus $a_{i} = \alpha^{i-1} d_P^{0, i-1}|_{Z^0_{N-1}(P^\bullet)}$ for $1 \le i \le N-1$.

Dually, by commutativity of the bottom rows, $b_{i}$ factors through $Z^{0}_{i}(Q^\bullet)$, which by acyclicity of $Q^\bullet$ is equal to $B^{0}_i(Q^\bullet)$.  By projectivity of $I_{i}$, we obtain a map $\beta^{i-1}: I_{i} \rightarrow Q^{i-N}$ such that $b_{i} = d_Q^{i-N, -i+N} \beta^{i-1}$.

Define $h^i = \beta^i \alpha^i: P^i \rightarrow Q^{i-N+1}$ for $1 \le i \le N-1$.  Then we have
\begin{eqnarray*}
f^{0}|_{Z^{0}_{N-1}(P^\bullet)} &=& \sum_{i=0}^{N-2} b_{i+1}a_{i+1} = \sum_{i=0}^{N-2} d_Q^{\circ, -i+N-1} h^{i} d_P^{0, i}|_{Z^0_{N-1}(P^\bullet)}\\
 &=& S_h(0, 0, N-1)|_{Z^{0}_{N-1}(P^\bullet)}
\end{eqnarray*}
Thus $(h^0, \cdots, h^{N-2})$ is a seed of $f^\bullet$ at $0$.
\end{proof}

%%%%%%%%%%%%%%%%%%%%%%%%%%%

If $(h^n, \cdots, h^{n+N-1})$ is a homotopy of $f^\bullet: P^\bullet \rightarrow Q^\bullet$ at $n$, it is clear that the shortened tuple $(h^n, \cdots, h^{n+N-2})$ is a seed at $n$, since the last term of $f^n = S_h(n, 0, N)$ vanishes on $Z^{n}_{N-1}(P^\bullet)$.  The next lemma establishes a converse.

\begin{lemma}
\label{seed lemma 2}
Let $f^\bullet: P^\bullet \rightarrow Q^\bullet$ be a morphism in $K_N^{ac}(\Proj{A})$.  Suppose there exists a seed $(h^n, \cdots, h^{n+N-2})$ of $f^\bullet$ at $n$.  Then there exists $h^{n+N-1}$ such that:\\
$\bullet$  $(h^n, \cdots, h^{n+N-1})$ is a homotopy at $n$.\\
$\bullet$  $(h^{n+1}, \cdots, h^{n+N-1})$ is a seed at $n+1$.

There also exists $h^{n-1}$ such that:\\
$\bullet$  $(h^{n-1}, h^n, \cdots, h^{n+N-2})$ is a homotopy at $n-1$.\\
$\bullet$  $(h^{n-1}, h^n, \cdots, h^{n+N-3})$ is a seed at $n-1$.
\end{lemma}

\begin{proof}
Let $\psi = f^n - S_h(n, 0, N-1)$.  Since $(h^n, \cdots, h^{n+N-2})$ is a seed at $n$, we have $\psi \mid_{Z^n_{N-1}(P^\bullet)}=0$, hence $\psi$ factors through $P^n/Z^{n}_{N-1}(P^\bullet)$.  By injectivity of $Q^n$, we obtain
\begin{eqnarray*}
\begin{tikzcd}
P^n \arrow[dr, "\psi", swap] \arrow[r, two heads] & P^n/Z^n_{N-1}(P^\bullet) \arrow[d, "\overline{\psi}"] \arrow[r, hook, "\overline{d_P^{n, N-1}}"] & P^{n+N-1} \arrow[dl, dashed, "h^{n+N-1}"] \\
& Q^n & 
\end{tikzcd}
\end{eqnarray*}
Thus
\begin{eqnarray*}
f^n &=& S_h(n, 0, N-1) + \psi = S_h(n, 0, N-1) + h^{n+N-1}d_P^{n, N-1}\\
 &=& S_h(n, 0, N)
\end{eqnarray*}
so $(h^n, \cdots, h^{n+N-1})$ is a homotopy at $n$.

To see that $(h^{n+1}, \cdots, h^{n+N-1})$ is a seed at $n+1$, note that
$$f^{n+1}d_P^n = d_Q^n f^n = d_Q^n S_h(n, 0, N) = S_h(n+1, 0, N-1)d_P^n$$
Since $d_P^n:  P^n \twoheadrightarrow Z^{n+1}_{N-1}(P^\bullet)$ is an epimorphism, we can cancel it on the right to obtain $f^{n+1}|_{Z^{n+1}_{N-1}(P^\bullet)} = S_h(n+1, 0, N-1)|_{Z^{n+1}_{N-1}(P^\bullet)}$, as desired.  

To construct $h^{n-1}$, let $\varphi = f^{n-1} - S_h(n-1, 1, N)$.  Note that
\begin{eqnarray*}
d_Q^{n-1} \varphi &=& d_Q^{n-1}f^{n-1} - d_Q^{n-1} S_h(n-1, 1, N)\\
 &=& (f^n - S_h(n, 0, N-1))d_P^{n-1} = 0
\end{eqnarray*}
where the last equality holds because $(h^n, \cdots, h^{n+N-1})$ is a seed at $n$.  Thus $\varphi$ factors through $Z^{n-1}_1(Q^\bullet)$, and by projectivity of $P^{n-1}$ we obtain
\begin{eqnarray*}
\begin{tikzcd}
& P^{n-1} \arrow[dl, dashed, "h^{n-1}", swap] \arrow[d, "\varphi"] &\\
Q^{n-N} \arrow[r, two heads, "d_Q^{\circ, N-1}", swap] & Z^{n-1}_1(Q^\bullet) \arrow[r, hook] & Q^{n-1}
\end{tikzcd}
\end{eqnarray*}
Thus
\begin{eqnarray*}
f^{n-1} &=& \varphi + S_h(n-1, 1, N) = d_Q^{\circ, N-1}h^{n-1} + S_h(n-1, 1, N)\\
 &=& S_h(n-1, 0, N)
\end{eqnarray*}
hence $(h^{n-1}, \cdots, h^{n+N-2})$ is a homotopy at $n-1$.  It follows immediately that $(h^{n-1}, \cdots, h^{n+N-3})$ is a seed at $n-1$.
\end{proof}

%%%%%%%%%%%%%%%%%%%%%%%%%%%

We are now ready to prove the main theorem of this section.

\begin{theorem}
\label{acyclic-stable equivalence}
$\overline{F}: K_N^{ac}(\Proj{\A}) \rightarrow \stab{N}{\A}$ is an equivalence.
\end{theorem}

\begin{proof}
Let $f^\bullet: P^\bullet \rightarrow Q^\bullet$ be a morphism in $K_N^{ac}(\Proj{\A})$ such that $\overline{F}(f) = 0$.
By Lemmas \ref{seed lemma 1} and \ref{seed lemma 2}, we can inductively define maps $h^i: P^i \rightarrow Q^{i-N+1}$ for all $i \in \Z$ such that $(h^n, \cdots, h^{n+N-1})$ is a homotopy at $n$ for every $n \in \Z$.  Thus $f$ is null-homotopic, and so $\overline{F}$ is faithful.

$\overline{F}$ is defined via a commutative diagram of functors
\begin{eqnarray*}
\begin{tikzcd}
C_N^{ac}(\Proj{\A}) \arrow[r, "F"] \arrow[d] & \MMor_{N-2}(\A) \arrow[d]\\
K_N^{ac}(\Proj{\A}) \arrow[r, "\overline{F}"] & \stab{N}{\A}
\end{tikzcd}
\end{eqnarray*}
By Propositions \ref{fullness on complexes} and \ref{essential surjectivity on complexes}, $F$ is full and essentially surjective, and the same is clearly true for the projection $\MMor_{N-2}(\A) \rightarrow \stab{N}{A}$.  It follows immediately that $\overline{F}$ is full and essentially surjective, hence an equivalence.
\end{proof}

%%%%%%%%%%%%%%%%%%%%%%%%%%%%%%%%%%%%%%%%%%%%%%%%%%%%%%%%%%%%%%%%%%%%%%%%%%%%%%%%%%%%%%%%%%%%%%%%%%%%%%%%%%%%%%%%%%%%%%%%%%%%%%%%%%%%%%%%%%%%%%%%%%%%%%%%%%%

\section{The $N$-Singularity Category}
\label{The N-Singularity Category}

Throughout this section, let $\A$ be an abelian category which is Frobenius exact.

There is a fully faithful additive functor $G:  \Mor_{N-2}(\A) \hookrightarrow C^b_N(\A)$ given by interpreting the object $(X_\bullet, \alpha_\bullet) \in \Mor_{N-2}(\A)$ as an $N$-complex concentrated in degrees $1$ through $N-1$.  In this section, we shall show that $G$ induces an equivalence $\overline{G}$ between $\stab{N}{\A}$ and $D^s_N(\A)$.

\begin{proposition}
\label{singularity functor}
$G$ induces a functor $\overline{G}: \stab{N}{\A} \rightarrow D^s_N(\A)$ of triangulated categories.
\end{proposition}

\begin{proof}
Let $G'$ denote the composition $$\MMor_{N-2}(\A) \xhookrightarrow{G} C^b_N(\A) \rightarrow D^b_N(\A) \rightarrow D^s_N(\A)$$

By Proposition \ref{inj/proj characterization}, $G$ maps projective objects in $\MMor_{N-2}(\A)$ to perfect complexes, hence $G'$ sends projective objects to zero.  Thus $G'$ induces an additive functor $\overline{G}: \stab{N}{\A} \rightarrow D^s_N(\A)$.

If $X_\bullet \rightarrowtail Y_\bullet \twoheadrightarrow Z_\bullet$ is admissible in $\MMor_{N-2}(\A)$, apply $G$ to obtain a short exact sequence in $C^b_{N}(\A)$.  By \cite[Proposition 3.7]{iyama2017derived}, there is a corresponding distinguished triangle $G(X_\bullet) \rightarrow G(Y_\bullet) \rightarrow G(Z_\bullet) \rightarrow \Sigma G(X_\bullet)$ in $D^b_N(\A)$, hence in $D^s_N(\A)$.

Consider an admissible exact sequence $X_\bullet \rightarrowtail I_{X\bullet} \twoheadrightarrow \Omega^{-1}X_\bullet$, with $I_{X\bullet}$ injective.  This induces a triangle $G(X_\bullet) \rightarrow 0 \rightarrow G(\Omega^{-1} X_\bullet) \xrightarrow{\phi_X} \Sigma G(X_\bullet)$ in $D^s_N(\A)$, which defines a natural isomorphism $\phi:  \overline{G}\Omega^{-1} \xrightarrow{\sim} \Sigma \overline{G}$.  Since every distinguished triangle in $\stab{N}{\A}$ is isomorphic to one arising from an admissible short exact sequence in $\MMor_{N-2}(\A)$, it follows easily that $(\overline{G}, \phi)$ is a triangulated functor.
\end{proof}

The functor $G$ also gives a canonical embedding of $\Mor_{N-2}(\A)$ into $D^b_{N-2}(\A)$.  With some extra hypotheses on $\A$, this is a corollary of \cite[Theorem 4.2]{iyama2017derived}; however, the proof below is valid for an arbitrary abelian category (which need not be Frobenius exact).

\begin{proposition}
\label{canonical embedding}
The composition $\Mor_{N-2}(\A) \xrightarrow{G} C^b_N(\A) \rightarrow D^b_N(\A)$ is fully faithful.
\end{proposition}

\begin{proof}
Let $(X_\bullet, \alpha_\bullet)$, $(Y_\bullet, \beta_\bullet) \in \Mor_{N-2}(\A)$.

To prove fullness, take a morphism $h:  G(X_\bullet) \rightarrow G(Y_\bullet)$ in $D^b_N(\A)$.  Write $h$ as the span $G(X_\bullet) \xleftarrow{s^\bullet} M^\bullet \xrightarrow{g^\bullet} G(Y_\bullet)$, where $s^\bullet$ is a quasi-isomorphism.  Since $G(X_\bullet)$ is concentrated in degrees $1$ through $N-1$, the natural map $\iota^\bullet: \sigma_{\le N-1}M^\bullet \hookrightarrow M^\bullet$ is also a quasi-isomorphism; thus $h$ can be written as $G(X_\bullet) \xleftarrow{s^\bullet \iota^\bullet} \sigma_{\le N-1}M^\bullet \xrightarrow{g^\bullet \iota^\bullet} G(Y_\bullet)$.  Let $f_\bullet: X_\bullet \rightarrow Y_\bullet$ be given by $f_i = H^i_{N-i}(g^\bullet) \circ H^i_{N-i}(s^\bullet)^{-1}$.

To see that $f_\bullet$ defines a morphism in $\Mor_{N-2}(\A)$, consider for each $1 \le i \le N-1$ the commutative diagrams
\begin{eqnarray}
\label{homology squares}
\begin{tikzcd}[column sep=small]
Z^i_{N-i}(M^\bullet) \arrow[d, "s^i \iota^i"] \arrow[r, two heads, "\pi^i"] & H^i_{N-i}(M^\bullet) \arrow[d, "H^i_{N-i}(s^\bullet)", swap] \arrow[d, "\sim"]\\
Z^i_{N-i}(G(X_\bullet)) \arrow[r, two heads] & H^i_{N-i}(G(X_\bullet))
\end{tikzcd},
\begin{tikzcd}[column sep=small]
Z^i_{N-i}(M^\bullet) \arrow[d, "g^i \iota^i"] \arrow[r, two heads, "\pi^i"] & H^i_{N-i}(M^\bullet) \arrow[d, "H^i_{N-i}(g^\bullet)", swap]\\
Z^i_{N-i}(G(Y_\bullet)) \arrow[r, two heads] & H^i_{N-i}(G(Y_\bullet))
\end{tikzcd}
\end{eqnarray}
Note that $Z^i_{N-i}(G(X_\bullet)) = H^i_{N-i}(G(X_\bullet)) = X_i$, and similarly for $Y_i$.  Thus the lower morphisms in both diagrams are just the identity maps on $X_i$ and $Y_i$.  In particular, $s^i\iota^i$ is an epimorphism.  We also have that
\begin{eqnarray}
\label{roof equation}
f_i \circ s^i \iota^i = H^i_{N-i}(g^\bullet) H^i_{N-i}(s^\bullet)^{-1} \circ s^i \iota^i = H^i_{N-i}(g^\bullet)\pi^i = g^i \iota^i
\end{eqnarray}
It follows that, for $1 \le i < N-1$,
\begin{eqnarray*}
f_{i+1} \alpha_i \circ s^i \iota^i = f_{i+1}s^{i+1}\iota^{i+1} d_M^i = g^{i+1}\iota^{i+1} d_M^i = \beta_i g^i \iota^i = \beta_i f_i \circ s^i \iota^i
\end{eqnarray*}
Since $s^i \iota^i$ is an epimorphism, we conclude that $f_{i+1} \alpha_i = \beta_i f^i$, hence $f_\bullet$ is a morphism.  From Equation (\ref{roof equation}) it follows immediately that $h = G(f_\bullet)$ in $D^b_{N}(\A)$.  Thus the functor is full.

To prove faithfulness, let $f_\bullet: X_\bullet \rightarrow Y_\bullet$ be such that $G(f_\bullet) = 0$ in $D^b_N(\A)$.  Then there is a quasi-isomorphism $s^\bullet: M^\bullet \rightarrow G(X_\bullet)$ such that $G(f_\bullet) s^\bullet = 0$ in $K^b_N(\A)$.  Define as above the quasi-isomorphism $\iota^\bullet: \sigma_{\le N-1}M^\bullet \hookrightarrow M^\bullet$; it follows that $G(f_\bullet) s^\bullet \iota^\bullet = 0$ in $K^b_N(\A)$.  Since $G(Y_\bullet)$ is concentrated in degrees $1$ through $N-1$, it is easily checked that the only null-homotopic morphism of complexes from $\sigma_{\le N-1}M^\bullet$ to $G(Y_\bullet)$ is the zero map.  Thus $G(f_\bullet) s^\bullet \iota^\bullet = 0$ in $C^b_N(\A)$; that is, $f_i s^i \iota^i = 0$ for all $1 \le i \le N-1$.

Note that the left square in (\ref{homology squares}) remains valid for all $1 \le i \le N-1$.  In particular, $s^i \iota^i: Z^i_{N-i}(M^\bullet) \twoheadrightarrow X_i$ is an epimorphism.  Thus $f_i = 0$ for all $i$.  Since $f_\bullet = 0$, the functor is faithful.
\end{proof}

We shall prove the following theorem via a sequence of lemmas.

\begin{theorem}
\label{stable-singularity equivalence}
Let $\A$ be an abelian category which is Frobenius exact.  Then $\overline{G}: \stab{N}{\A} \rightarrow D^s_N(\A)$ is an equivalence.
\end{theorem}

First, it will be helpful to more easily express morphisms in $D_N(\A)$.  The following proposition is completely analogous to the known result for $N=2$.  It holds for any abelian category and does not require the hypothesis of Frobenius exactness.

\begin{lemma}
\label{morphism simplification}
Let $X^\bullet \in K_N(\A)$, $P^\bullet \in K^-_N(\Proj{\A})$, $I^\bullet \in K^+_N(\Inj{\A})$.  Let $f: P^\bullet \rightarrow X^\bullet$ and $g: X^\bullet \rightarrow I^\bullet$ be morphisms in $D_N(\A)$.  Then $f$ and $g$ can be represented by morphisms in $K_N(\A)$.  
\end{lemma}

\begin{proof}
Express $f$ as the span $P^\bullet \xleftarrow{p^\bullet} Q^\bullet \xrightarrow{h^\bullet} X^\bullet$, where $p^\bullet$ is a quasi-isomorphism.  Then $p^\bullet$ fits into a triangle $\Sigma^{-1}C^\bullet \rightarrow Q^\bullet \xrightarrow{p^\bullet} P^\bullet \rightarrow C^\bullet$ in $K_N(\A)$, where $C^\bullet$ is an acyclic $N$-complex.  By \cite[Lemma 3.3]{iyama2017derived}, $\Hom_{K_N(\A)}(P^\bullet, C^\bullet) = 0$.  Since the last map in the above triangle is zero, the map $p^\bullet$ admits a section $s^\bullet: P^\bullet \rightarrow Q^\bullet$ in $K_N(\A)$.  It follows that the span representing $f$ is equivalent to $P^\bullet \xleftarrow{id} P^\bullet \xrightarrow{h^\bullet s^\bullet} X^\bullet$, hence $f$ is equal to the morphism of complexes $h^\bullet s^\bullet$.

Similarly, express $g$ as a cospan $X^\bullet \xrightarrow{e^\bullet} J^\bullet \xleftarrow{i^\bullet} I^\bullet$, where $i^\bullet$ is a quasi-isomorphism.  Extend $i^\bullet$ to the triangle $D^\bullet \rightarrow I^\bullet \xrightarrow{i^\bullet} J^\bullet \rightarrow \Sigma D^\bullet$ in $K_N(\A)$, for some acyclic $D^\bullet$.  Again by \cite[Lemma 3.3]{iyama2017derived}, there are no nonzero morphisms from $D^\bullet$ to $I^\bullet$, hence $i^\bullet$ admits a retraction $r^\bullet$ in $K_N(\A)$.  Thus $g$ is equal to the span $X^\bullet \xrightarrow{r^\bullet e^\bullet} I^\bullet \xleftarrow{id} I^\bullet$, hence $g = r^\bullet e^\bullet$.
\end{proof}

\begin{lemma}
\label{G faithful}
$\overline{G}$ is faithful.
\end{lemma}

\begin{proof}
Let $(X_\bullet, \alpha_\bullet), (Y_\bullet, \beta_\bullet) \in \stab{N}{\A}$, and let $f_\bullet : X_\bullet \rightarrow Y_\bullet$ be a fixed representative of a morphism in $\stab{N}{\A}$.  Suppose $\overline{G}(f_\bullet) = 0$.

We first show that $f_\bullet$ factors through an object in $\Mor_{N-2}(\Proj{\A})$.  Since $\overline{G}(f_\bullet) = 0$ in $D^s_N(\A)$, there exists a morphism with perfect cone $s: \overline{G}(Y_\bullet) \rightarrow M^\bullet$ in $D^b_N(\A)$ such that $s \circ \overline{G}(f_\bullet) = 0$.  Let $I^\bullet$ denote the cocone of $s^\bullet$; we obtain a morphism of triangles in $D^b_N(\A)$:
\begin{eqnarray*}
\begin{tikzcd}
0 \arrow[r] \arrow[d] & \overline{G}(X_\bullet) \arrow[r, "id"] \arrow[d, dashed, "g^\bullet"] & \overline{G}(X_\bullet) \arrow[r] \arrow[d, "\overline{G}(f_\bullet)"] & 0 \arrow[d]\\
\Sigma^{-1}M^\bullet \arrow[r] & I^\bullet \arrow[r, "h^\bullet"] & \overline{G}(Y_\bullet) \arrow[r, "s"] & M^\bullet
\end{tikzcd}
\end{eqnarray*}

Changing the bottom row up to isomorphism, we may assume that $I^\bullet$ is a bounded complex of projective-injectives; by Lemma \ref{morphism simplification} we have that $h^\bullet$ and $g^\bullet$ can be chosen to be morphisms of complexes.  Note that $h^ig^i: X_i \rightarrow Y_i$ defines a morphism in $\MMor_{N-2}(\A)$ whose image under $G$ is $h^\bullet g^\bullet$.  Since $\overline{G}(f_\bullet) = h^\bullet g^\bullet$ in $D^b_N(\A)$, by Proposition \ref{canonical embedding} we have that $f_i = h^i g^i$ for all $1 \le i \le N-1$.

We shall now construct $(I'_\bullet, \iota_\bullet) \in \Proj{\MMor_{N-2}(\A)}$ and a factorization $X_\bullet \xrightarrow{\hat{g}_\bullet} I'_\bullet \xrightarrow{\hat{h}_\bullet} Y_\bullet$ of $f_\bullet$.  Define $I'_i := \bigoplus_{j=1}^i I^j = I'_{i-1} \oplus I^i$, and let $\iota_i: I'_i \hookrightarrow I'_i \oplus I^{i+1}$ be given by $\begin{bmatrix} id \\ d_I^i\pi_i \end{bmatrix}$, where $\pi_i : I_i' \twoheadrightarrow I^i$ is the canonical projection.  It is clear that $(I'_\bullet, \iota_\bullet) \in \Proj{\MMor_{N-2}(\A)}$, since each $I'_i$ is projective-injective and each $\iota_i$ is a (necessarily split) monomorphism.  Define $\hat{h}_\bullet: I'_\bullet \rightarrow Y_\bullet$ by $\hat{h}_i := h^i \pi_i$; it is straightforward to check that $\hat{h}_\bullet$ is a morphism in $\MMor_{N-2}(\A)$.  

We shall inductively construct a family $\hat{g}_i: X_i \rightarrow I'_i$ such that $\pi_i \hat{g}_i = g^i$ for all $1 \le i \le N-1$ and $\iota_{i-1} \hat{g}_{i-1} = \hat{g}_{i}\alpha_{i-1}$ for all $2 \le i \le N-1$.  Let $\hat{g}_1 = g^1$; note that $\pi_1: I'_1 \twoheadrightarrow I^1$ is the identity map, so the desired equation holds.  Next, suppose that $\hat{g}_{i-1}$ has been constructed; by injectivity of $I'_{i-1}$ we may lift $\hat{g}_{i-1}$ to $\phi_i: X_i \rightarrow I'_{i-1}$ such that $\hat{g}_{i-1} = \phi_i \alpha_{i-1}$.  Define $\hat{g}_i: X_i \rightarrow I'_{i-1} \oplus I^i$ to be $\begin{bmatrix} \phi_i \\ g^{i} \end{bmatrix}$; it easy to verify that $\hat{g}_i$ satisfies both of the desired equations.  Thus the morphism $\hat{g}_\bullet: X_\bullet \rightarrow I'_\bullet$ is defined.  Furthermore, we have that $\hat{h}_i \hat{g}_i = h^i\pi_i\hat{g}_i = h^ig^i = f_i$, hence $f_\bullet = \hat{h}_\bullet \hat{g}_\bullet$.  Thus $f_\bullet = 0$ in $\stab{N}{\A}$ and $\overline{G}$ is faithful.
\end{proof}

To prove fullness, we need a better understanding of how to express morphisms in $D^s_N(\A)$.

\begin{lemma}
\label{fullness lemma 1}
Let $(X_\bullet, \alpha_\bullet), (Y_\bullet, \beta_\bullet) \in \MMor_{N-2}(\A)$.  Then the natural map $\Hom_{D^b_N(\A)}(\overline{G}(X_\bullet), \overline{G}(Y_\bullet)) \rightarrow \Hom_{D^s_N(\A)}(\overline{G}(X_\bullet), \overline{G}(Y_\bullet))$ is surjective.  That is, any morphism $\overline{G}(X_\bullet) \rightarrow \overline{G}(Y_\bullet)$ in $D^s_N(\A)$ can be represented by a span of the form
$$\overline{G}(X_\bullet) \xleftarrow{id} \overline{G}(X_\bullet) \xrightarrow{g} \overline{G}(Y_\bullet)$$
where $g$ is a morphism in $D^b_N(\A)$.
\end{lemma}

\begin{proof}
Any morphism in $\Hom_{D^s_N(\A)}(\overline{G}(X_\bullet), \overline{G}(Y_\bullet))$ can be represented by a span $\overline{G}(X_\bullet) \xleftarrow{s} M^\bullet \xrightarrow{f} \overline{G}(Y_\bullet)$, where $s$ and $f$ are morphisms in $D^b_N(\A)$ and $s$ fits into a triangle $M^\bullet \xrightarrow{s} \overline{G}(X_\bullet) \xrightarrow{t} I^\bullet \rightarrow \Sigma M^\bullet$ with $I^\bullet \in D^{perf}_N(\A)$.  Changing $I^\bullet$ up to isomorphism in $D^b_N(\A)$, we may assume without loss of generality that $I^\bullet$ is a bounded $N$-complex of projective-injectives.  By Lemma \ref{morphism simplification} we can represent $t$ by a morphism of complexes $t^\bullet$.  Changing $M^\bullet$ up to isomorphism in $D^b_N(\A)$, we can also assume that $M^\bullet$ is the cocone of $t^\bullet$ in $K^b_N(\A)$, hence $M^\bullet \xrightarrow{s^\bullet} \overline{G}(X_\bullet) \xrightarrow{t^\bullet} I^\bullet \rightarrow \Sigma M^\bullet$ is a triangle in $K^b_N(\A)$.  Note that if $I^\bullet = 0$, then $s^\bullet: M^\bullet \xrightarrow{\sim} \overline{G}(X_\bullet)$ is an isomorphism in $K^b_N(\A)$ and we are done; we thus assume that $I^\bullet$ is nonzero.

By Theorem \ref{acyclic-stable equivalence}, there exists an acyclic $N$-complex $P^\bullet$ of projective-injectives such that $X_\bullet = Z^0_\bullet(P^\bullet)$.  Let $\hat{X}^\bullet$ be the $N$-complex
$$\hat{X}^\bullet = 0 \rightarrow X_1 \hookrightarrow X_2 \hookrightarrow \cdots \hookrightarrow X_{N-1} \hookrightarrow P^0 \rightarrow P^1 \rightarrow \cdots$$
where $X_1$ is in degree $1$.  It is straightforward to check that $\hat{X}^\bullet$ is acyclic.  For any integer $m\ge N$, there is a natural morphism of $N$-complexes $p^\bullet: \tau_{\le m}\hat{X}^\bullet \twoheadrightarrow \overline{G}(X_\bullet)$.  We claim that for sufficiently large $m \ge N$, there is a morphism of $N$-complexes $r^\bullet : \tau_{\le m} \hat{X}^\bullet \rightarrow M^\bullet$ satisfying $p^\bullet = s^\bullet r^\bullet$, and an equivalence of morphisms in $D^s_N(\A)$:
\begin{eqnarray*}
\overline{G}(X_\bullet) \xleftarrow{s^\bullet} M^\bullet \xrightarrow{f} \overline{G}(Y_\bullet) 
=
\overline{G}(X_\bullet) \xleftarrow{p^\bullet} \tau_{\le m}\hat{X}^\bullet \xrightarrow{fr^\bullet} \overline{G}(Y_\bullet) 
\end{eqnarray*}

Let $k$ be the maximum integer such that $I^k$ is nonzero, and choose $m \ge max(N, k+N)$.  We have a triangle in $K^+_N(\A)$
$$\tau_{>m}\hat{X}^\bullet \rightarrow \hat{X}^\bullet \rightarrow \tau_{\le m} \hat{X}^\bullet \rightarrow \Sigma\tau_{>m}\hat{X}^\bullet$$
arising from the chain-wise split exact sequence of complexes.  All nonzero terms of $\tau_{>m}\hat{X}^\bullet$ and $\Sigma \tau_{>m}\hat{X}^\bullet$ occur in degrees greater than $k$, hence $\Hom_{K^+_N(\A)}(\tau_{> m}\hat{X}^\bullet, I^\bullet) = 0 = \Hom_{K^+_N(\A)}(\Sigma \tau_{> m}\hat{X}^\bullet, I^\bullet)$.  Since $\hat{X}^\bullet$ is acyclic, $\Hom_{K^+_N(\A)}(\hat{X}^\bullet, I^\bullet) = 0$ by \cite[Lemma 3.3]{iyama2017derived}.  Applying the functor $\Hom_{K^+_N(\A)}(-, I^\bullet)$ to the triangle, we see that $\Hom_{K^b_N(\A)}(\tau_{\le m}\hat{X}^\bullet, I^\bullet) = 0$.
 
The kernel of $p^\bullet$ is $J^\bullet := \tau_{\le m} ((\tau_{\ge 0} P^\bullet)[-N]) \in K^b_N(\Proj{\A})$; the chainwise split exact sequence $\begin{tikzcd}[column sep=small] J^\bullet \arrow[r, hook] & \tau_{\le m} \hat{X}^\bullet \arrow[r, two heads, "p^\bullet"] & \overline{G}(X_\bullet) \end{tikzcd}$ induces a triangle in $K^b_N(\A)$.  Since $\Hom_{K^b_N(\A)}(\tau_{\le m}\hat{X}^\bullet, I^\bullet) = 0$, we obtain a morphism of triangles in $K^b_N(\A)$:
\begin{eqnarray*}
\begin{tikzcd}
J^\bullet \arrow[r] \arrow[d, dashed, "\Sigma^{-1} q^\bullet"] & \tau_{\le m} \hat{X}^\bullet \arrow[r, "p^\bullet"] \arrow[d] & \overline{G}(X_\bullet) \arrow[r] \arrow[d, "t^\bullet"] & \Sigma J^\bullet \arrow[d, dashed, "q^\bullet"]\\
\Sigma^{-1} I^\bullet \arrow[r] & 0 \arrow[r] & I^\bullet \arrow[r, "id"] & I^\bullet
\end{tikzcd}
\end{eqnarray*}
which in turn yields
\begin{eqnarray*}
\begin{tikzcd}
J^\bullet \arrow[r] \arrow[d, "\Sigma^{-1} q^\bullet"] & \tau_{\le m} \hat{X}^\bullet \arrow[r, "p^\bullet"] \arrow[d, dashed, "r^\bullet"] & \overline{G}(X_\bullet) \arrow[r] \arrow[d, "id"] & \Sigma J^\bullet \arrow[d, "q^\bullet"]\\
\Sigma^{-1} I^\bullet \arrow[r] & M^\bullet \arrow[r, "s^\bullet"] &\overline{G}(X_\bullet) \arrow[r, "t^\bullet"] & I^\bullet
\end{tikzcd}
\end{eqnarray*}

Since $s^\bullet$ and $p^\bullet = s^\bullet r^\bullet$ both have perfect cones, it follows from the octahedron axiom that $r^\bullet$ does as well.  The desired equivalence of roofs $f(s^\bullet)^{-1} = (fr^\bullet)(s^\bullet r^\bullet)^{-1} = (fr^\bullet)(p^\bullet)^{-1}$ follows immediately.

Furthermore, since $J^\bullet \in K^b_N(\Proj{\A})$ is concentrated in degrees $N$ through $m$ and $\overline{G}(Y_\bullet)$ is concentrated in degrees $1$ through $N-1$, \linebreak[4] $\Hom_{K^b_N(\A)}(J^\bullet, \overline{G}(Y_\bullet)) = 0 = \Hom_{D^b_N(\A)}(J^\bullet, \overline{G}(Y_\bullet))$.  We obtain a morphism of triangles in $D^b_N(\A)$:
\begin{eqnarray*}
\begin{tikzcd}
J^\bullet \arrow[r] \arrow[d] & \tau_{\le m} \hat{X}^\bullet \arrow[r, "p^\bullet"] \arrow[d, "fr^\bullet"] & \overline{G}(X_\bullet) \arrow[r] \arrow[d, dashed, "g"] & \Sigma J^\bullet \arrow[d]\\
0 \arrow[r] & \overline{G}(Y_\bullet) \arrow[r, "id"] &\overline{G}(Y_\bullet) \arrow[r] & 0
\end{tikzcd}
\end{eqnarray*}
Therefore we have an equivalence of morphisms
\begin{eqnarray*}
\overline{G}(X_\bullet) \xleftarrow{p^\bullet} \tau_{\le m}\hat{X}^\bullet \xrightarrow{fr^\bullet} \overline{G}(Y_\bullet) 
=
\overline{G}(X_\bullet) \xleftarrow{id} \overline{G}(X_\bullet) \xrightarrow{g} \overline{G}(Y_\bullet) 
\end{eqnarray*}
\end{proof}

\begin{corollary}
\label{G full}
$\overline{G}$ is full.
\end{corollary}

\begin{proof}
Let $X_\bullet, Y_\bullet \in \stab{N}{\A}$, and let $g: \overline{G}(X_\bullet) \rightarrow \overline{G}(Y_\bullet)$ be a morphism in $D^s_N(\A)$.  By Lemma \ref{fullness lemma 1}, $g$ can be taken to be a morphism in $D^b_N(\A)$, and by Proposition \ref{canonical embedding}, $g = G(f_\bullet)$ for some $f_\bullet: X_\bullet \rightarrow Y_\bullet$ in $\MMor_{N-2}(\A)$.  Let $\overline{f}_\bullet$ denote the image of $f_\bullet$ in $\stab{N}{\A}$.  By the construction of $\overline{G}$, $\overline{G}(\overline{f}_\bullet) = G(f_\bullet) = g$.  Thus $\overline{G}$ is full. 
\end{proof}

It remains to show that $\overline{G}$ is essentially surjective.  Recall the objects $\chi_i(X)_\bullet \in \MMor_{N-2}(\A)$ of Definition \ref{fundamental objects}.  We shall also use the formula in \cite[Lemma 2.6]{iyama2017derived} describing the action of $\Sigma$ on the complexes $\mu^s_r(X)$ in the homotopy category.

\begin{lemma}
$\overline{G}$ is essentially surjective, hence an equivalence of triangulated categories.
\end{lemma}

\begin{proof}
By Proposition \ref{singularity functor}, Lemma \ref{G faithful} and Corollary \ref{G full}, $\overline{G}$ is a fully faithful functor of triangulated categories, hence its essential image $Im(\overline{G})$ is a triangulated subcategory of $D^s_N(\A)$.

Let $\mathcal{S} = \{\mu^k_i(X) \mid k \in \Z, 1 \le i \le N-1, X \in \A\}$, and let $\mathcal{T}$ denote the smallest triangulated subcategory of $D^s_N(\A)$ containing $\mathcal{S}$.  We claim that $\mathcal{T} = D^s_N(\A)$.  Suppose for a contradiction that $X^\bullet \in D^s_N(\A)$ is a bounded $N$-complex of minimum possible length such that $X^\bullet \notin \mathcal{T}$.  Clearly $X^\bullet \neq 0$; suppose $m$ is the largest integer such that $X^m \neq 0$.  Then we have a natural short exact sequence of complexes which induces a triangle $\mu^m_1(X^m) \rightarrow X^\bullet \rightarrow \tau_{< m}X^\bullet \rightarrow \Sigma \mu^m_1(X^m)$ in $D^s_N(\A)$.  Note that $\mu^m_1(X^m) \in \mathcal{T}$ by definition and $\tau_{<m}X^\bullet \in \mathcal{T}$ since it has length less than $X^\bullet$.  It follows that $X^\bullet \in \mathcal{T}$, a contradiction.  Thus $\mathcal{T} = D^s_N(\A)$.  

We claim $\mathcal{S}$ is contained in $Im(\overline{G})$; once this is proved, it follows immediately that $Im(\overline{G}) = \mathcal{T} = D^s_N(\A)$, hence $\overline{G}$ is an equivalence.

We first show that $\mathcal{S}' = \{\mu^k_i(X) \mid 1 \le i \le k \le N-1, X \in \A\}$, consisting of all elements of $\mathcal{S}$ which are concentrated in degrees $1$ through $N-1$, is contained in $Im(\overline{G})$.  Fix $X \in \A$.  It is immediate that $\mu^{N-1}_i(X) = \overline{G}(\chi_i(X))$ for each $1 \le i \le N-1$.  For $1 \le i \le k \le N-1$, we have a short exact sequence of $N$-complexes $\mu^{N-1}_{N-1-k}(X) \hookrightarrow \mu^{N-1}_{N-1-k+i}(X) \twoheadrightarrow \mu^{k}_i(X)$ which induces a triangle in $D^s_N(\A)$.  Since the first two members of this triangle lie in $Im(\overline{G})$, so does $\mu^k_i(X)$.  Thus $\mathcal{S}' \subseteq Im(\overline{G})$.

For any $\mu^k_i(X) \in \mathcal{S}$, there is a unique $x \in \Z$ such $k = xN + r$, where $0 \le r < N$.  Then $\Sigma^{2x} \mu^k_i(X) \cong \mu^k_i(X)[xN] = \mu^r_i(X)$.  If $i \le r$, then $\mu^r_i(X) \in \mathcal{S}'$.  Otherwise, $0 \le r < i$, hence $\Sigma^{-1}(\mu^r_i(X)) = \mu^{N-(i-r)}_{N-i}(X) \in \mathcal{S}'$.  In either case, $\Sigma^y \mu^k_i(X) \in Im(\overline{G})$ for some value of $y$, hence $\mu^k_i(X) \in Im(\overline{G})$.  Thus $\mathcal{S} \subseteq Im(\overline{G})$, hence $\overline{G}$ is essentially surjective.
\end{proof}

%%%%%%%%%%%%%%%%%%%%%%%%%%%%%%%%%%%%%%%%%%%%%%%%%%%%%%%%%%%%%%%%%%%%%%%%%%%%%%%%%%%%%%%%%%%%%%%%%%%%%%%%%%%%%%%%%%%%%%%%%%%%%%%%%%%%%%%%%%%%%%%%%%%%%%

\section{Calabi-Yau Properties of $\stab{N}{\rmod{A}}$}
\label{Calabi-Yau Properties}

In this section we let $A$ be an associative algebra over a field $F$.  We shall assume that $A$ is finite-dimensional and self-injective.  Fix an integer $N \ge 2$.  Under these hypotheses, the category $\rmod{A}$ is Frobenius exact, hence $\stab{N}{\rmod{A}}$ (hereafter abbreviated as $\stab{N}{A}$) is a triangulated category by Theorem \ref{frobenius}.

It is known that $\stab{N}{A}$ possesses a Serre functor.  (See \cite{ringel2008auslander} for case $N=3$ and \cite{xiong2014auslander} for general $N$.)  The goal of this section is to obtain a sufficient condition for $\stab{N}{A}$ to be fractionally Calabi-Yau.  In order to obtain a useful description of the Serre functor on $\stab{N}{A}$, we must first introduce several other functors.

%%%%%%%%%%%%%%%%%%%%%%%%%%%%%%%%%%%
\subsection{The Minimal Monomorphism Functor}

The \textbf{minimal monomorphism} construction was introduced in \cite{ringel2008auslander} for $N=3$ and \cite{zhang2011monomorphism} for general $N$.  To simplify notation in this section, we shall let $k = N-2$.

\begin{definition}
Let $(X_\bullet, \alpha_\bullet) \in \Mor_{k}(A)$.  Define $(\Mimo_\bullet(X), m_\bullet(X)) \in \MMor_{k}(A)$ as follows.  For $1 \le i \le k$, let $ker(\alpha_i) \hookrightarrow J_{i+1}(X)$ denote the injective hull of $ker(\alpha_i)$, and choose a lift $\omega_i: X_i \rightarrow J_{i+1}(X)$ of this map.  Let $J_1(X) = 0$.  For $1 \le i \le k+1$, let $I_i(X) := \bigoplus_{j=1}^{i}J_j(X)$, so that $I_1(X) = 0$ and $I_i(X) = J_i(X) \oplus I_{i-1}(X)$.  Define $\Mimo_i(X) := X_i \oplus I_i(X)$ and let $m_i(X): \Mimo_i(X) \rightarrow \Mimo_{i+1}(X)$ be given by $$m_i(X) := \begin{bmatrix} \alpha_i & 0\\ \omega_i & 0  \\ 0 & 1\end{bmatrix}:  X_i \oplus I_i(X) \hookrightarrow X_{i+1} \oplus J_{i+1}(X) \oplus I_i(X)$$

Given $f_\bullet: X_\bullet \rightarrow Y_\bullet$, define $\Mimo_\bullet(f):  \Mimo_\bullet(X) \rightarrow \Mimo_\bullet(Y)$ inductively as follows.  Define $\Mimo_1(f) := f_1: X_1 \rightarrow Y_1$.  Suppose that we have defined $\Mimo_{i-1}(f): X_{i-1} \oplus I_{i-1}(X) \rightarrow Y_{i-1} \oplus I_{i-1}(Y)$ to be of the form $\begin{bmatrix} f_{i-1} & 0 \\ \phi_{i-1} & \psi_{i-1} \end{bmatrix}$.  Define $\begin{bmatrix} \phi_i & \psi_i \end{bmatrix}: X_i \oplus I_i(X) \rightarrow I_i(Y)$ to be a lift of the map
\begin{eqnarray*}
\begin{tikzcd}
X_{i-1} \oplus I_{i-1}(X) \arrow[r, "\Mimo_{i-1}(f)"] &[10pt] Y_{i-1} \oplus I_{i-1}(Y) \arrow[r, hook, "m_{i-1}(Y)"] & Y_i \oplus I_i(Y) \arrow[r, two heads] &[-10pt] I_i(Y)
\end{tikzcd}
\end{eqnarray*}
 along the injection $m_{i-1}(X): X_{i-1} \oplus I_{i-1}(X) \hookrightarrow X_i \oplus I_i(X)$.  Then define $\Mimo_i(f): X_i \oplus I_i(X) \rightarrow Y_i \oplus I_i(Y)$ by the matrix $\begin{bmatrix} f_i & 0 \\ \phi_i & \psi_i \end{bmatrix}$.
\end{definition}

In the above definition, it is clear that each $m_i(X)$ is a monomorphism, and that the map $\Mimo_\bullet(f)$ is a morphism in $\MMor_k(A)$.  Note also that we have a morphism $\Mimo_\bullet(X) \twoheadrightarrow X_\bullet$ given by component-wise projection onto $X_\bullet$.  We now state some basic properties of this construction.

\begin{proposition}
\label{Mimo properties}
1)  For any object $X_\bullet \in \Mor_k(A)$, $\Mimo_\bullet(X)$ is independent, up to isomorphism in $\MMor_k(A)$, of the choice of the maps $\omega_i$.
\\
2)  For any morphism $f_\bullet: X_\bullet \rightarrow Y_\bullet$ in $\Mor_k(A)$, the image of $\Mimo_\bullet(f)$ in $\stab{N}{A}$ is independent of the choice of maps $\phi_i$ and $\psi_i$.
\\
3)  $\Mimo$ acts as the identity on both objects and morphisms in $\stab{N}{A}$.
\\
4)  $\Mimo$ defines a functor $\Mor_k(A) \rightarrow \stab{N}{A}$ which descends to functors $\underline{\Mor_k}(A) \rightarrow \stab{N}{A}$ and $\overline{\Mor_k}(A) \rightarrow \stab{N}{A}$.
\\
5)  $\Mimo:  \underline{\Mor_k}(A) \rightarrow \stab{N}{A}$ is right adjoint to the inclusion functor.
\end{proposition}

\begin{proof}
1)  It is proved in \cite[Lemma 2.3]{zhang2011monomorphism} that the projection $\Mimo_\bullet(X) \twoheadrightarrow X_\bullet$ is a right minimal approximation of $X_\bullet$ in $\MMor_k(A)$, hence is unique up to isomorphism in $\MMor_k(A)$.  In particular, any two choices of the maps $\omega_i$ in the construction of $\Mimo_\bullet(X_\bullet)$ yield isomorphic objects.
\\\\
2)  Given $f_\bullet: X_\bullet \rightarrow Y_\bullet$ and two different choices in the construction of $\Mimo_\bullet(f)$, it is easy to check that their difference factors through the projective-injective object $I_1(Y) \hookrightarrow I_2(Y) \hookrightarrow \cdots \hookrightarrow I_{k+1}(Y)$.
\\\\
3)  If $X_\bullet \in \MMor_k(A)$, then $ker(\alpha_i) = 0$ for all $i$.  Thus $I_i(X) = 0$ and $Mimo_\bullet(X) = X_\bullet$.  The statement about morphisms is immediate.
\\\\
4)  The first statement is easily verified.  For the second statement, note that by Propositions \ref{inj/proj characterization} and \ref{basic inj/proj characterization} the projective objects of $\Mor_k(A)$ are precisely the projective-injective objects of $\MMor_k(A)$, hence are preserved by $\Mimo$.  Thus the functor $\Mimo: \Mor_k(A) \rightarrow \stab{N}{A}$ kills projectives and so descends to $\underline{\Mor_k}(A)$.  Similarly, the injective objects in $\Mor_k(A)$ are component-wise projective-injective with all maps split epimorphisms; such objects are mapped to projective-injective objects by $\Mimo$, hence $\Mimo$ also descends to $\overline{\Mor_k}(A)$.
\\\\
5)  Let $\iota: \stab{N}{A} \hookrightarrow \underline{\Mor_k}(A)$ denote the inclusion functor.  Let $X_\bullet \in \underline{\Mor_k}(A), Y_\bullet \in \stab{N}{A}$.  Define natural transformations
\begin{align*}
\epsilon: \iota \circ \Mimo \rightarrow 1_{\underline{\Mor_k}(A)} & & \eta: 1_{\stab{N}{A}} \rightarrow \Mimo \circ \iota
\end{align*}
as follows.  Let $\epsilon_{X_\bullet}: \Mimo_\bullet(X) \rightarrow X_\bullet$ be the component-wise projection onto $X_\bullet$, and let $\eta_{Y_\bullet}: Y_\bullet \rightarrow \Mimo_\bullet(Y) = Y_\bullet$ be the identity map.  It follows immediately from definitions that $\epsilon$ and $\eta$ are indeed natural transformations; it remains to verify that they satisfy the triangle identities.

That $(\epsilon \iota) \circ (\iota \eta) = id_\iota$ is immediate.  To see that $(\Mimo\epsilon)\circ(\eta\Mimo) = id_{\Mimo}$, evaluate at $X_\bullet$ and note that the left-hand side simplifies to $\Mimo_\bullet(\epsilon_X): \Mimo_\bullet(X) \rightarrow \Mimo_\bullet(X)$.  We can choose this map to be the identity map.  Thus the pair $(\iota, \Mimo)$ is an adjunction.
\end{proof}

%%%%%%%%%%%%%%%%%%%%%%%%%%%%%%%%%%%
\subsection{Cokernel and Rotation Functors}

Throughout this section, we shall let $k = N-2$ to simplify notation.

\begin{definition}
For $(X_\bullet, \alpha_\bullet) \in \MMor_k(A)$, define
$$\Cok_\bullet(X) := X_{k+1} \twoheadrightarrow coker(\alpha^{k}_1) \twoheadrightarrow coker(\alpha^{k-1}_2) \twoheadrightarrow \cdots \twoheadrightarrow coker(\alpha_{k})$$
For $f_\bullet: X_\bullet \rightarrow Y_\bullet$, let $\Cok_\bullet(f): \Cok_\bullet(X) \rightarrow \Cok_\bullet(Y)$ be given by the component-wise induced maps on the cokernels.
\end{definition}

It is clear that $\Cok$ defines a functor $\MMor_k(A) \rightarrow \Mor_k(A)$ which sends projective-injective objects to injective objects.  Thus $\Cok$ descends to a functor $\stab{N}{A} \rightarrow \overline{\Mor_k}(A)$.  Though we shall not need this fact, we note that $\Cok$ also defines an exact equivalence between $\MMor_k(A)$ and $\EMor_k(A)$ which descends to a triangulated equivalence between the respective stable categories.

\begin{definition}
Define the \textbf{rotation functor} to be the composition $$R = \Mimo \circ \Cok: \stab{N}{A} \rightarrow \overline{\Mor_k}(A) \rightarrow \stab{N}{A}$$
\end{definition}

The rotation construction was first defined in \cite{ringel2008auslander} for $N =3$ and later generalized to arbitrary $N$ in \cite{xiong2014auslander}.  Our formulation differs slightly in that it is defined on $\stab{N}{A}$ rather than $\Mor_{N-2}(\stab{}{A})$.  On $\stab{N}{A}$, the rotation functor can be somewhat difficult to work with, but it simplifies considerably when expressed in terms of complexes.

Recall the triangulated equivalence $\overline{G}: \stab{N}{A} \rightarrow D^s_{N}(A)$ defined in Proposition \ref{singularity functor}.  Note that $\overline{G}$ extends to a functor $\overline{\Mor_k}(A) \rightarrow D^s_N(A)$.

\begin{proposition}
\label{rotation compatibility}
There is an isomorphism $\Sigma[-1] \circ \overline{G} \cong \overline{G} \circ R$ of functors $\stab{N}{A} \rightarrow D^s_N(A)$.
\end{proposition}

\begin{proof}
Let $(X_\bullet, \alpha_\bullet) \in \stab{N}{A}$.  The short exact sequence in $C^b_N(A)$
$$\overline{G}(X_\bullet) \hookrightarrow \mu^{N-1}_{N}(X_{N-1}) \twoheadrightarrow \overline{G}(\Cok_\bullet(X))[1]$$ induces a triangle in $D^s_N(A)$.  The middle term is null-homotopic, so we have an isomorphism $\overline{G}(\Cok_\bullet(X))[1] \xrightarrow{\sim} \Sigma(\overline{G}(X_\bullet))$ in $D^s_N(A)$; since the above exact sequence is natural in $X_\bullet$, so is this isomorphism.  Applying $[-1]$ yields a natural isomorphism $\overline{G} \circ \Cok \cong \Sigma [-1] \circ \overline{G}$.

Applying $\overline{G}$ to the short exact sequence in $\MMor_k(A)$
$$I_\bullet(\Cok(X)) \rightarrowtail \Mimo_\bullet(\Cok(X)) \twoheadrightarrow \Cok_\bullet(X)$$
we obtain a triangle in $D^s_N(A)$.  The left term is mapped to $D^{perf}_N(A)$, hence vanishes; we obtain an isomorphism $\overline{G}R(X_\bullet) \cong \overline{G}(\Cok_\bullet(X))$ which is clearly natural in $X_\bullet$.  Thus $\overline{G}\circ R \cong \overline{G} \circ \Cok \cong \Sigma[-1] \circ \overline{G}$.
\end{proof}

%%%%%%%%%%%%%%%%%%%%%%%%%%%%%%%%%%%
\subsection{Upper Triangular Matrices}

Throughout this section, we shall let $n = N-1$ to simplify notation.

Let $B = T_{n}(A)$ denote the $F$-algebra of $n \times n$ upper-triangular matrices with entries in $A$.  We write $E_{i,j}$ for the matrix with $1_A$ in the $(i,j)$-th position (that is, row $i$ and column $j$) and $0$'s everywhere else.

Given $X\in \rmod{B}$, we can create the following object in $\Mor_{n-1}(A)$:
\begin{eqnarray*}
\begin{tikzcd}
XE_{1,1} \arrow[r, "r_{E_{1,2}}"] & XE_{2,2} \arrow[r, "r_{E_{2,3}}"] & \cdots \arrow[r, "r_{E_{n-1,n}}"] &  XE_{n,n}
\end{tikzcd}
\end{eqnarray*}
More explicitly, there is an equivalence $M_r:  \rmod{B} \xrightarrow{\sim} \Mor_{n-1}(A)$ given by $M_r(X) = (XE_{\bullet, \bullet}, r_{E_{\bullet, \bullet+1}})$.  The inverse of $M_r$ is given by $M^{-1}_r(X_\bullet, f_\bullet) = \bigoplus_{i=1}^n X_i$, where $E_{i,i}$ acts as projection onto the $i$-th coordinate and $E_{i, i+j}$ acts as $f_i^j$.

Similarly, there is an equivalence $M_l: \lmod{B} \xrightarrow{\sim} \Mor_{n-1}(A^{op})$ which is given by $M_l(X) = (E_{n+1-\bullet, n+1-\bullet}X, l_{E_{n-\bullet, n+1-\bullet}})$.  Its inverse is given by $M^{-1}_l(X_\bullet, f_\bullet) = \bigoplus_{i=1}^n X_i$, where $E_{i,i}$ acts as projection onto $X_{n+1-i}$ and $E_{i-j,i}$ acts as $f_{n+1-i}^j$.

It is easy to check that $M_r(B) \cong \bigoplus_{i=1}^n \chi_i(A)_\bullet \cong M_l(B)$ has injective dimension $1$ in $\Mor_{n-1}(A)$, hence $B$ is Gorenstein.  (Recall the definition of $\chi_i(A)_\bullet$ from Section \ref{Projective and Injective Objects}.)  The following proposition allows us to identify the monomorphism category of $A$ with the Gorenstein projective $B$-modules.

\begin{proposition}
The functors $M_r$ and $M_l$ restrict to the following exact equivalences:\\
1)  $M_r:  \Gproj{B} \xrightarrow{\sim} \MMor_{n-1}(A)$\\
2)  $M_l:  \Gproj{B^{op}} \xrightarrow{\sim} \MMor_{n-1}(A^{op})$\\
3)  $M_r:  \Ginj{B} \xrightarrow{\sim} \EMor_{n-1}(A)$\\
4)  $M_l:  \Ginj{B^{op}} \xrightarrow{\sim} \EMor_{n-1}(A^{op})$

Each of the above equivalences descends to a triangulated equivalence between the respective stable categories.
\end{proposition}

\begin{proof}  It is clear that $M_r$ and $M_l$ are exact equivalences.  Once 1)-4) have been established, it is also clear that the functors descend to triangulated equivalences between the stable categories.  All that is needed is to show that each functor has the appropriate image.
\\\\
1)  Let $(X_\bullet, \alpha_\bullet) \in \Mor_{n-1}(A)$.  Since $M_r(B) \cong \bigoplus_{i=1}^n \chi_i(A)_\bullet$, it suffices to prove that $X_\bullet \in \MMor_{n-1}(A)$ if and only if $\Ext^1(X_\bullet, \chi_i(A)_\bullet) = 0$ for all $1 < i \le n$.  (Since $\chi_1(A)_\bullet$ is injective, $\Ext^1(X_\bullet, \chi_1(A)_\bullet) = 0$ for any $X_\bullet$.)  Let $\overline{\chi}_i(A)_\bullet$ denote the cokernel of the natural inclusion $\chi_i(A)_\bullet \hookrightarrow \chi_1(A)_\bullet$.  Define a complex in $C^b(\Mor_{n-1}(A))$ 
\begin{eqnarray*}
I^\bullet(i) = \begin{tikzcd}[column sep=small] \cdots \arrow[r] & 0 \arrow[r] & \chi_1(X)_\bullet \arrow[r, two heads] & \overline{\chi}_i(X)_\bullet \arrow[r] & 0 \arrow[r] & \cdots \end{tikzcd}
\end{eqnarray*}
with $\chi_1(X)_\bullet$ in degree $0$.  $I^\bullet(i)$ is an injective resolution of $\chi_i(A)_\bullet$, hence $\Ext^1(X_\bullet, \chi_i(A)_\bullet) = \Hom_{K^b(\MMor_{n-1}(A))}(X_\bullet, I^\bullet(i)[1])$.  Note that a morphism of complexes $X_\bullet \rightarrow I^\bullet(i)[1]$ is the same data as a morphism $f_{i-1}: X_{i-1} \rightarrow A$; such a morphism is null-homotopic if and only if $f_{i-1}$ factors through $\alpha_{i-1}^j$ for all $1 \le j \le n-i+1$.  

Suppose $X_\bullet \in \MMor_{n-1}(A)$.  Since $\alpha_{i-1}^j$ is a monomorphism and $A$ is injective, any morphism $f_{i-1}: X_{i-1} \rightarrow A$ admits a factorization $f_{i-1} = g_{i-1+j}\alpha_{i-1}^{j}$, hence $\Ext^1(X_\bullet, \chi_i(A)_\bullet) = 0$.  Conversely, if $\alpha_{i-1}$ is not injective for some $1 < i \le n$, then there is a nonzero morphism $ker(\alpha_{i-1}) \rightarrow A$ which can be lifted to a morphism $f_{i-1}: X_{i-1} \rightarrow A$.  Since $f_{i-1}$ is nonzero on $ker(\alpha_{i-1})$, it cannot factor through $\alpha_{i-1}$, hence $f_{i-1}$ defines a nonzero element of $\Ext^1(X_\bullet, \chi_{i}(A)_\bullet)$.  Thus $M_r$ identifies $\Gproj{B}$ with $\MMor_{n-1}(A)$.
\\\\
2)  Since $M_l(B) \cong \bigoplus_{i=1}^n \chi_i(A)_\bullet$, the proof is identical to 1).
\\\\
3)  By Proposition \ref{duality compatibility} below, $M_r \cong D_*M_l D$.  The result then follows from 2).
\\\\
4)  The result follows from Proposition \ref{duality compatibility} and 1).
\end{proof}

%%%%%%%%%%%%%%%%%%%%%%%%%%%%%%%%%%%%%

\subsection{Duality and the Nakayama Functor}
In this section, we continue to write $n = N-1$.

Given a covariant functor $F: \rmod{A} \rightarrow \mathcal{C}$, there is an induced functor $F_*: \Mor_{n-1}(A) \rightarrow \Mor_{n-1}(\mathcal{C})$ given by $F(X_\bullet, \alpha_\bullet) = (F(X_\bullet), F(\alpha_\bullet))$.  Given a contravariant functor $G: (\rmod{A})^{op} \rightarrow \mathcal{C}$, we likewise obtain a functor $G_*:  \Mor_{n-1}(A)^{op} \rightarrow \Mor_{n-1}(\mathcal{C})$, this time given by $G_*(X_\bullet, \alpha_\bullet) = (G(X_{n+1-\bullet}), G(\alpha_{n-\bullet}))$.

%Let $D = \Hom_F(-, F)$ denote the $F$-linear duality.  This functor is defined on all relevant module categories, and we shall use the same notation for all variations.  We also have an induced functor $D_*: (\Mor_{n-1}(\rmod{A}))^{op} \rightarrow \Mor_{n-1}(\lmod{A})$ given by $D_*(X_\bullet, f_\bullet) = (DX_{n+1-\bullet}, f^*_{n-\bullet})$.

Recall the Nakayama functor $\nu_A$, defined in Section \ref{Gorenstein Algebras} to be the composition of the dualities $D = \Hom_F(-,F)$ and $\Hom_A(-,A)$.  Note that both of the induced functors $D_*$ and $\Hom_A(-,A)_*$ define dualities $\Mor_{n-1}(A)^{op} \xrightarrow{\sim} \Mor_{n-1}(A^{op})$ which identify the monomorphism subcategory with the epimorphism subcategory, and vice versa.  It follows that the equivalence $\nu_{A*} = D_* \Hom_{A}(-, A)_*: \Mor_{n-1}(A) \xrightarrow{\sim} \Mor_{n-1}(A)$, preserves both $\MMor_{n-1}(A)$ and $\EMor_{n-1}(A)$ and descends to the corresponding stable categories.

In contrast with the behavior of $\nu_{A*}$, recall that $\nu_B$ restricts to an equivalence $\Gproj{B} \xrightarrow{\sim} \Ginj{B}$; it is therefore worth investigating the relationship between these two functors.  Before we express $\nu_B$ in the language of the monomorphism category, it will be helpful to first translate the $F$-linear duality on $B$.

\begin{proposition}
\label{duality compatibility}
There is an isomorphism $D_* \circ M_l \cong M_r \circ D$ of functors $(\lmod{B})^{op} \rightarrow \Mor_{n-1}(A)$.  Similarly, $M_l \circ D \cong D_* \circ M_r$.
\end{proposition}

\begin{proof}
Let $X \in \lmod{B}$.  The left $A$-module map $l_{E_{i,i}}: X \twoheadrightarrow E_{i,i}X$ yields a monomorphism $l_{E_{i,i}}^*: D(E_{i,i}X) \hookrightarrow DX$ whose image is $(DX)E_{i,i}$.  We have a commutative diagram in $\rmod{A}$.
\begin{eqnarray*}
\begin{tikzcd}
D(E_{i-1, i-1}X) \arrow[r, "l_{E_{i-1,i}}^*"]  \arrow[d, swap, "\sim"] \arrow[d, "l_{E_{i-1,i-1}}^*"]& D(E_{i,i}X) \arrow[d, "l_{E_{i,i}}^*"] \arrow[d, swap, "\sim"]\\
(DX)E_{i-1, i-1} \arrow[r, "r_{E_{i-1,i}}"] & (DX)E_{i,i}
\end{tikzcd}
\end{eqnarray*}
hence $l_{E_{\bullet, \bullet}}^*: D_*M_l(X) \xrightarrow{\sim} M_rD(X)$ is an isomorphism which is easily verified to be natural in $X$.

The second isomorphism follows immediately by precomposing with $D$ and postcomposing with $D_*$.
\end{proof}

\begin{proposition}
\label{Nakayama compatibility}
There is an isomorphism $M_r \circ \nu_{B} \cong \Cok\nu_{A*} \circ M_r$ of functors $\Gproj{B} \rightarrow \EMor_{n-1}(A)$.
\end{proposition}

\begin{proof}
It is enough to show that $D_*M_r \nu_{B} \cong D_*\Cok\nu_{A*} M_r$.  By Proposition \ref{duality compatibility}, we have that
$$D_*M_r \nu_{B} \cong M_l D\nu_{B} \cong M_l\Hom_{B}(-, B)$$
Since $\nu_A$ is exact, it is easily verified that $\Cok \nu_{A*} \cong \nu_{A*}\Cok$, hence 
$$D_*\Cok\nu_{A*} M_r \cong D_*\nu_{A*}\Cok M_r \cong \Hom_A(-,A)_*\Cok M_r$$
It thus suffices to construct $\zeta:  M_l\Hom_{B}(-,B) \xrightarrow{\sim} \Hom_A(-, A)_*\Cok M_r$, an isomorphism of functors $\Gproj{B}^{op} \rightarrow \MMor_{n-1}(A^{op})$.

Let $X \in \Gproj{B}$.  Note that $E_{i,i}\Hom_{B}(X, B)$ consists of precisely those homomorphisms with image in $E_{i,i}B$.  Thus 
$$M_l\Hom_{B}(X,B) = (\Hom_{B}(X, E_{n+1-\bullet, n+1-\bullet}B), l_{E_{n-\bullet, n+1-\bullet}})$$
A direct computation shows that 
$$\Hom_A(-,A)_*\Cok M_r(X) = (\Hom_A(XE_{n,n}/XE_{n-\bullet, n}, A), \pi_{n-\bullet}^*)$$
where $\pi_i: XE_{n,n}/XE_{i-1,n} \twoheadrightarrow XE_{n,n}/XE_{i,n}$ is the canonical projection.  (Here we define $XE_{0,n}$ to be $0$.)

Given $f \in \Hom_{B}(X, E_{i,i}B)$, note that the restriction of $f$ to $XE_{n,n}$ has image in $E_{i,i}BE_{n,n} = E_{i,n}B = E_{i,n}A$, which is canonically isomorphic to $A$ as an $(A,A)$-bimodule.  Furthermore, $f(XE_{i-1,n}) \subseteq E_{i,i}BE_{i-1, n} = 0$, hence the restriction descends to a map
$$\overline{f\mid_{XE_{n,n}}}: XE_{n,n}/XE_{i-1, n} \rightarrow E_{i,n}A \cong A$$  Let $\zeta_{X,i}:  \Hom_{B}(X, E_{i,i}B) \rightarrow \Hom_A(XE_{n,n}/XE_{i-1,n}, A)$ be the map sending $f$ to $\overline{f\mid_{XE_{n,n}}}$.

To show that $\zeta_{X,i}$ is injective, let $f\in ker(\zeta_{X,i})$ and let $x \in X$.  Since $\zeta_{X,i}(f) = 0$, then $f(XE_{n,n}) = 0$ and so $f(x)E_{j,n} = f(xE_{j,n}E_{n,n}) = 0$ for all $j \le n$.  The map $r_{E_{j,n}}: BE_{j,j} \hookrightarrow BE_{n,n}$ is injective for all $j\le n$; it follows from the above equation that $f(x)E_{j,j} = 0$ for all $j \le n$, hence $f(x) = 0$.  Thus $f = 0$ and $\zeta_{X,i}$ is injective.

To see that $\zeta_{X,i}$ is surjective, take any $g\in \Hom_A(XE_{n,n}/XE_{i-1,n}, A)$.  Define $f: X \rightarrow E_{i,i}B$ by $f(x) = \sum_{j=i}^n g(xE_{j,n})E_{i,j}$.  A direct computation shows that for any $1 \le r \le s \le n$,
$$f(xE_{r,s}) = \sum_{j=i}^n g(xE_{r,n})E_{i,s} = f(x)E_{r,s}$$
It follows that $f$ is a right $B$-module morphism and $\zeta_{X,i}(f) = g$.  Thus $\zeta_{X, i}$ is an isomorphism for each $i$.

It is easily checked that $\zeta_{X, n+1-\bullet}$ is a morphism in $\MMor_{n-1}(A^{op})$ and is natural in $X$, hence the two functors are isomorphic.
\end{proof}

%%%%%%%%%%%%%%%%%%%%%%%%%%%%%%%%%

\subsection{Serre Duality}

The inclusion functor $\underline{\Gproj{B}} \hookrightarrow \underline{\rmod{B}}$ possesses a right adjoint $P: \underline{\rmod{B}} \rightarrow \underline{\Gproj{B}}$ \cite[Lemma 6.3.7]{krause2021homological}.  We have already seen that $\Mimo$ plays an analogous role in the monomorphism category, so it is no surprise that the two functors are related.

\begin{proposition}
\label{P compatibility}
There is an isomorphism $M_r \circ P \cong \Mimo \circ M_r$ of functors $\underline{\rmod{B}} \rightarrow \stab{N}{A}$.
\end{proposition}

\begin{proof}
Let $\iota_1: \stab{N}{A} \hookrightarrow \underline{\Mor_{n-1}}(A)$ and $\iota_2: \underline{\Gproj{B}} \hookrightarrow \underline{\rmod{B}}$ be the inclusion functors.  It is clear that $\iota_1 M_r = M_r \iota_2$.  By Proposition \ref{Mimo properties}, $\Mimo$ is right adjoint to $\iota_1$; it follows that both $P$ and $M_r^{-1} \Mimo M_r$ are right adjoint to $\iota_2$, hence $P \cong M_r^{-1} \Mimo M_r$.  The result follows.
\end{proof}

We are ready to describe the Serre functors on $\stab{N}{A}$ and $D^s_N(A)$.  We shall write $\Omega_A$, $\Omega_B$, and $\Omega_N$ to denote the syzygy functors on $\stab{}{A}$, $\stab{}{B}$ and $\stab{N}{A}$, respectively.  Recall that since $A$ is self-injective, $\nu_A$ is exact and so lifts to $D^s_N(A)$.

\begin{theorem}
\label{Serre}
$\Omega_N R \nu_{A_*}$ is a Serre functor on $\stab{N}{A}$.  $[-1]\nu_{A}$ is a Serre functor on $D^s_N(A)$.
\end{theorem}
\begin{proof}
By \cite[Corollary 6.4.10]{krause2021homological}, $\underline{\Gproj{B}}$ has Serre functor $S := \Omega_{B} P \nu_{B}$.  Thus $M_r S M_r^{-1}$ is a Serre functor for $\stab{N}{A}$ and $\overline{G}M_r S M_r^{-1}\overline{G}^{-1}$ is a Serre functor for $D^s_N(A)$.  Then
\begin{align*}
M_r S  M_r^{-1} &= M_r \Omega_{B} P \nu_{B} M_r^{-1}\\
&\cong \Omega_N M_r P \nu_{B} M_r^{-1}\\
&\cong \Omega_N \Mimo M_r \nu_{B} M_r^{-1} \tag{Proposition \ref{P compatibility}}\\
&\cong \Omega_N \Mimo \Cok \nu_{A*} \tag{Proposition \ref{Nakayama compatibility}}\\
&= \Omega_N R \nu_{A*}
\end{align*}
and
\begin{align*}
\overline{G}M_r S M_r^{-1}\overline{G}^{-1} &\cong \overline{G} \Omega_N R \nu_{A*} \overline{G}^{-1}\\
&\cong \Sigma^{-1} \overline{G}R \nu_{A*} \overline{G}^{-1}\\
&\cong \Sigma^{-1} \Sigma[-1] \overline{G} \nu_{A*} \overline{G}^{-1} \tag{Proposition \ref{rotation compatibility}}\\
&\cong [-1]\nu_{A}
\end{align*}
where the isomorphism $\overline{G} \nu_{A*} \cong \nu_{A} \overline{G}$ follows immediately from exactness of $\nu_A$.
\end{proof}

\begin{corollary}
\label{CYdim}
Suppose the Nakayama automorphism of $A$ has order $r$.  Let $s = lcm(N, r)$ and $t = \frac{s}{N}$.  If $N>2$, then $\stab{N}{A}$ is $(-2t, s)$-Calabi-Yau.  $\stab{}{A}$ is $(-1, r)$-Calabi-Yau.
\end{corollary}

\begin{proof}
It suffices to check that $D^s_N(A)$ has the appropriate Calabi-Yau property.  We have that $\nu_A^r \cong id$, hence $\nu_A^s \cong id$.  Then
$$([-1]\nu_A)^s \cong [-s] = [-tN] \cong \Sigma^{-2t}$$
For $N=2$, we have $\Sigma = [1]$, hence $([-1]\nu_A)^r \cong \Sigma^{-r}$.
\end{proof}

\begin{corollary}
Suppose $A$ is symmetric.  Then $\stab{}{A}$ is $(-1)$-Calabi-Yau and $\stab{N}{A}$ is $(-2, N)$-Calabi-Yau for all $N > 2$.
\end{corollary}

\begin{proof}
Since $A$ is symmetric, $\nu_A = id$ hence $r = 1$.  The statement follows.
\end{proof}

The above integer pairs need not be minimal.  The presence of additional relations between the functors $\Omega$, $\nu_{A*}$ and $R$ may allow $\stab{N}{A}$ to be $(x,y)$-Calabi-Yau for smaller values of $x$ and $y$; see below for a concrete example.

%%%%%%%%%%%%%%%%%%%%%%%%%%%%%%

\subsection{An Example}

Let $F$ be any field, let $Q$ be the quiver $\begin{tikzcd} 1 \arrow[r, "\alpha", bend left] & 2 \arrow[l, "\beta", bend left] \end{tikzcd}$, and let $A = FQ/rad^2(FQ)$.  Then $A$ is self-injective with four indecomposable modules: the simple modules $S_1$ and $S_2$ and their two-dimensional injective hulls $I_1$ and $I_2$.

\begin{center}
\begin{tikzcd}[column sep = small]
& \left[I_1\right] \arrow[dr] & & \left[I_2 \right] \arrow[dr]& \\
S_1 \arrow[ur] & & S_2 \arrow[ur] & & S_1
\end{tikzcd}
\\
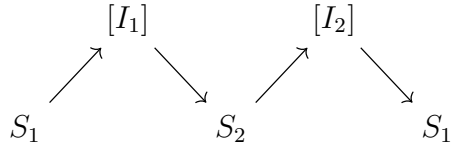
The Auslander-Reiten quiver of $A$.\\
Vertices in brackets are projective-injective.
\end{center}

Fix some $N\ge 2$.  For any integers $i, j \ge 0$ satisfying $1 \le i+j \le N-1$, define objects $X(i, j)$ and $Y(i,j)$ in $\MMor_{N-2}(A)$ by
\begin{eqnarray*}
X(i,j) := 0 \rightarrow \cdots \rightarrow 0 \rightarrow S_1 \rightarrow \cdots \rightarrow S_1 \rightarrow I_1 \rightarrow \cdots \rightarrow I_1\\
Y(i,j) := 0 \rightarrow \cdots \rightarrow 0 \rightarrow S_2 \rightarrow \cdots \rightarrow S_2 \rightarrow I_2 \rightarrow \cdots \rightarrow I_2
\end{eqnarray*}
Here each sequence has exactly $i$ simples and $j$ projective-injectives, and each morphism is the canonical inclusion.

In $\rmod{A}$, every monomorphism from an indecomposable module $M$ into a direct sum $Y \oplus Z$ factors through either $Y$ or $Z$, so $(M_\bullet, \alpha_\bullet) \in \MMor_{N-2}(A)$ is indecomposable if and only if each $M_i$ is indecomposable.  Thus the indecomposable objects of $\MMor_{N-2}(A)$ are precisely the $X(i,j)$ and $Y(i,j)$.  The indecomposable projective-injectives are precisely the objects $X(0, j)$ and $Y(0,j)$.

The Nakayama automorphism of $A$ has order $2$, so by Corollary \ref{CYdim}, $\stab{N}{A}$ is $(-4, 2N)$-Calabi-Yau if $N$ is odd and $(-2, N)$ if $N > 2$ is even.  However, it is easy to check that $\nu_{A*} \cong \Omega \cong \Omega^{-1}$ on $\stab{N}{A}$ for any $N$.  It follows from Proposition \ref{rotation compatibility} that $R$ and $\Omega^{-1}$ commute, since the corresponding functors $\Sigma$ and $\Sigma[-1]$ commute in $D^s_N(A)$.  Thus $\stab{N}{A}$ has Serre functor $S = \Omega R \nu_{A*} \cong R$, and $D^s_N(A)$ has Serre functor $\Sigma[-1]$.  In particular,
$$S^N \cong \Omega^{-N+2} \cong \begin{cases} \Omega^{-1} & N \text{ odd}\\  id & N \text{ even} \end{cases}$$
Thus for $N > 2$, $\stab{N}{A}$ is $(1,N)$-Calabi-Yau for odd $N$ and $(0, N)$-Calabi-Yau for even $N$.

A straightforward computation shows that for any $i>0$,
\begin{eqnarray*}
S(X(i,j)) &=& \begin{cases} Y(i, j-1) & j > 0 \\ X(N-i, i-1) & j = 0\end{cases}
\\
S(Y(i,j)) &=& \begin{cases} X(i, j-1) & j > 0 \\ Y(N-i, i-1) & j = 0\end{cases}
\\
\Omega (X(i,j)) &=& Y(i,j)
\\
\Omega (Y(i,j)) &=& X(i,j)
\end{eqnarray*}
It follows immediately that $S^n$ is not isomorphic to any power of $\Omega$ for any $0 < n < N$.

We conclude by providing the Auslander-Reiten quiver of $\MMor_{N-2}(A)$ for representative values of $N$.

\begin{landscape}
\begin{center}
\begin{tikzcd}[column sep = small]
&Y(1,0) \arrow[dr] \arrow[r, dashed]	 &\left[Y(0,1)\right] \arrow[r, dashed] &Y(1,1) \arrow[dr] \arrow[r, dashed] &\left[Y(0,2)\right] \arrow[r, dashed] & X(2,0) \arrow[dr] &\\
X(1,1)\arrow[ur] \arrow[r, dashed] &\left[X(0,2)\right] \arrow[r, dashed] &Y(2,0) \arrow[ur] & &X(1,0) \arrow[ur] \arrow[r, dashed] &\left[X(0,1)\right] \arrow[r, dashed] &X(1,1)
\end{tikzcd}
\\
The Auslander-Reiten quiver of $\MMor_{3}(A)$.\\
Vertices in brackets are projective-injective.
\end{center}
\vspace{3cm}
\begin{center}
\begin{tikzcd}[column sep = small]
Y(1,2) \arrow[r, dashed] \arrow[dr] &\left[Y(0,3)\right] \arrow[r, dashed] &X(3,0) \arrow[dr] & &Y(1,0) \arrow[r, dashed] \arrow[dr] &\left[Y(0,1)\right] \arrow[r, dashed] &Y(1,1) \arrow[r, dashed] \arrow[dr] &\left[Y(0,2)\right] \arrow[r, dashed] &Y(1,2)\\
&X(2,0) \arrow[ur] \arrow[dr] & &X(2,1) \arrow[ur] \arrow[dr] & &Y(2,0) \arrow[ur] \arrow[dr] & &Y(2,1) \arrow[ur] \arrow[dr] &\\
X(1,0) \arrow[r, dashed] \arrow[ur] &\left[X(0,1)\right] \arrow[r, dashed] &X(1,1) \arrow[r, dashed] \arrow[ur] &\left[X(0,2)\right] \arrow[r, dashed] &X(1,2) \arrow[r, dashed] \arrow[ur] &\left[X(0,3)\right] \arrow[r, dashed] &Y(3,0) \arrow[ur] & &X(1,0)
\end{tikzcd}
\\
The Auslander-Reiten quiver of $\MMor_{4}(A)$.\\
Vertices in brackets are projective-injective.
\end{center}
\end{landscape}

%%%%%%%%%%%%%%%%%%%%%%%%%%%%%%%%%%%%%%%%%%%%%%%%%%%%%%%%%%%%%%%%%%%%%%%%%%%%%%%%%%%%%%%%%%%%%%%%%%%%%%%%%%%%%%%%%%%%%%%%%%%%%%%%%%%%%%%%%%%%%%%%%%%%%%

\bibliography{References}{}
\bibliographystyle{plain}

\end{document}